\documentclass[10pt, letterpaper, twocolumn]{article}

\usepackage[utf8]{inputenc}
\usepackage{mathpazo}
\usepackage[T1]{fontenc}
\usepackage{microtype}
\usepackage[margin=0.8in]{geometry}
\usepackage{balance}

\usepackage{enumitem}
\setlist{topsep=0pt, itemsep=0pt}

\usepackage{graphicx} 
\usepackage{caption}
\usepackage{subcaption}

\usepackage{algorithm}
\usepackage{algpseudocode}
\algrenewcommand{\algorithmiccomment}[1]{\hfill\textit{\textcolor{gray}{\small$\triangleright$ #1}}}
\algnewcommand{\LineComment}[2]{\Statex\hspace{#1}\textit{\textcolor{gray}{\small$\triangleright$ #2}}\vspace{1pt}}

\usepackage{scalerel}
\newcommand{\medcap}{\raisebox{0.37ex}{$\mathbin{\scaleobj{0.75}{\,\displaystyle\bigcap\,}}$}}

\usepackage{xcolor}
\definecolor{beige}{RGB}{245, 245, 220}
\colorlet{DullPurple}{purple!30!blue!40!gray}
\colorlet{DullGray}{gray!75!black}
\colorlet{AquaMarine}{blue!50!green!60!gray}
\colorlet{PurpleBlue}{purple!20!blue}
\def\customred{red!50}
\def\customgreen{green!70!blue}
\def\customblue{blue!50}
\def\customorange{red!50!yellow!70}
\definecolor{trunkone}{rgb}{0.1,0.1,0.2}
\definecolor{trunktwo}{rgb}{0.3,0.3,0.4}
\definecolor{trunkthree}{rgb}{0.5,0.5,0.6}

\usepackage{hyperref}
\hypersetup{
    colorlinks=true,
    linkcolor=DullPurple,
    filecolor=DullPurple,      
    urlcolor=PurpleBlue,
    pdftitle={DDTO},
    pdfauthor={Purnanand Elango},
    citecolor=AquaMarine,
}

\usepackage[hang,flushmargin]{footmisc}
\newcommand\blfootnote[1]{%
  \begingroup
  \renewcommand\thefootnote{}\footnote{\noindent#1}%
  \addtocounter{footnote}{-1}%
  \endgroup
}

\usepackage{tikz}
\tikzset{every picture/.style={line width=1pt}}       
\usetikzlibrary{arrows.meta}
\usetikzlibrary{hobby}
\newcommand{\inlinesquare}[3][0.5]{
  \tikz \draw[line width={#2}, draw={#3}] (0,0) rectangle (#1,#1)%
  (-#1,#1/2) -- (2*#1,#1/2);
}
\newcommand{\inlinedot}[3][0.5]{
  \tikz \draw[line width={#2}, draw={#3}, fill={#3}] (0,0) circle [radius=#1]%
  (-2.5*#1,0) -- (2.5*#1,0);
}
\newcommand{\inlinecircle}[3][0.5]{
  \tikz \draw[line width={#2}, draw={#3}] (0,0) circle [radius=#1]%
  (-2.2*#1,0) -- (2.2*#1,0);
}
\newcommand{\inlinetriangle}[3][0.5]{
  \tikz \draw[line width={#2}, draw={#3}] (0,0) -- (#1,0) -- (0.5*#1,0.866*#1) -- cycle%
  (-#1/2,0.35*#1) -- (1.5*#1,0.35*#1);
}

\usepackage{amsmath,amssymb,bm}
\usepackage{centernot} 

\usepackage{amsthm}

\usepackage{thmtools}
\declaretheoremstyle[
  spaceabove=2pt,
  spacebelow=1pt,
  headfont=\color{black}\normalfont\bfseries,
  bodyfont=\normalfont\itshape,
  notefont=\color{DullGray}\normalfont\itshape\bfseries]{clrBlack}
\declaretheoremstyle[spaceabove=2pt,spacebelow=1pt,headfont=\color{DullGray}\normalfont\bfseries\itshape,qed=\qedsymbol]{clrGray}

\declaretheorem[style=clrBlack,name=Theorem]{theorem}
\declaretheorem[style=clrBlack,name=Proposition]{proposition}
\declaretheorem[style=clrBlack,name=Definition]{definition}
\declaretheorem[style=clrBlack,name=Corollary]{corollary}
\declaretheorem[style=clrBlack,name=Lemma]{lemma}
\declaretheorem[style=clrBlack,name=Remark]{remark}
\declaretheorem[style=clrBlack,name=Assumption]{assumption}
\declaretheorem[style=clrGray,name=Proof,numbered=no]{proof}

\usepackage{titlesec}
\titleformat*{\section}{\large\bfseries}
\titleformat*{\subsection}{\normalsize\bfseries}
\titleformat*{\subsubsection}{\small\bfseries}
\titleformat*{\paragraph}{\small\bfseries}
\titleformat*{\subparagraph}{\small\bfseries}

\titlespacing*{\section}
{0pt} 
{6pt plus 2pt minus 2pt} 
{2pt plus 1pt minus 1pt} 

\titlespacing*{\subsection}
{0pt} 
{6pt plus 2pt minus 2pt} 
{2pt plus 1pt minus 1pt} 

\newcommand{\behcetfull}{Beh\c{c}et~A\c{c}\i kme\c{s}e}

\let\emptyset\varnothing
\newcommand{\mb}[1]{\mathbb{#1}}
\newcommand{\mc}[1]{\mathcal{#1}}
\newcommand{\mr}[1]{\mathrm{#1}}
\newcommand{\bR}{\mb{R}}
\newcommand{\bRnx}{\mb{R}^{n_x}}
\newcommand{\bRnu}{\mb{R}^{n_u}}
\newcommand{\bX}{\mb{X}}
\newcommand{\bU}{\mb{U}}
\newcommand{\bZr}[2]{{\textstyle[#1\!:\!#2]}}
\newcommand{\cF}{\mc{F}}
\newcommand{\cB}{\mc{B}}
\newcommand{\cR}{\mc{R}}
\newcommand{\ol}[1]{\overline{#1}}
\newcommand{\ul}[1]{\underline{#1}}
\newcommand{\derv}[1]{\overset{\scriptscriptstyle\circ}{#1}}
\newcommand   {\ddto} {\textsc{ddto}}
\newcommand{\ddtoscp} {\textsc{ddto-scp}}
\newcommand{\ddtoqcvx}{\textsc{ddto-qcvx}}
\newcommand{\ddtomicp}{\textsc{ddto-micp}}
\newcommand{\gfunc}{\mathfrak{g}}
\newcommand{\tfinal}{t_{\mr{f}}}
\newcommand{\normplus}[1]{|#1|_+}
\newcommand{\selector}[1]{{}^{#1}\!E}
\newcommand{\zeros}[1]{0_{#1}}
\newcommand{\eye}[1]{I_{#1}}
\renewcommand{\dot}[1]{\overset{{\text{\large.}}}{#1}}
\newcommand{\znrm}[1]{\|#1\|_{\scriptscriptstyle\blacklozenge}}
\newcommand{\ctscvx}{{\scalebox{1.1}{\textsc{{\scalebox{0.73}{ct-}}sc{\scalebox{0.73}{vx}}}}}}

\title{\bf Deferred-Decision Trajectory Optimization}

\author{Purnanand Elango$^\ast$, Selahattin Burak Sars{\i}lmaz$^{\dagger}$, and \behcetfull$^\S$}

\date{}

\begin{document}

\twocolumn[
  \begin{@twocolumnfalse}
    \maketitle
    \begin{abstract}
        We present {\ddto}---deferred-decision trajectory optimization---a framework for trajectory generation with resilience to unmodeled uncertainties and contingencies. The key idea is to ensure that a collection of candidate targets is reachable for as long as possible while satisfying constraints, which provides time to quantify the uncertainties.  We propose optimization-based constrained reachability formulations and construct equivalent  cardinality minimization problems, which then inform the design of computationally tractable and efficient solution methods that leverage state-of-the-art convex solvers and sequential convex programming (SCP) algorithms. The goal of establishing the equivalence between constrained reachability and cardinality minimization is to provide theoretically-sound underpinnings for the proposed solution methods. We demonstrate the solution methods on real-world optimal control applications encountered in quadrotor motion planning.
    \end{abstract}
    \vspace{2em}
  \end{@twocolumnfalse}
]

\blfootnote{%
$^{\ast}$Research Scientist, Mitsubishi Electric Research Laboratories (MERL), Cambridge, MA, USA. Email: {\tt\small elango@merl.com}\\%
$^{\dagger}$Assistant Professor, Department of Electrical and Computer Engineering, Utah State University, Logan, UT, USA. Email: {\tt\small burak.sarsilmaz@usu.edu}\\%
$^{\S}$Professor, William E. Boeing Department of Aeronautics and Astronautics, University of Washington, Seattle, WA, USA. Email: {\tt\small behcet@uw.edu}\\%
This work was supported by the Office of Naval Research Grant N00014-20-1-2288. P. Elango was at the University of Washington during the development of this work.}

\section{Introduction}
The reference signals generated by the feedforward block of a control system often solve constrained optimal control problems \cite{tsiotras2017toward}. It is desirable to have the system response tightly track the reference signals in closed-loop with a feedback controller so that the constraint satisfaction and performance requirements guaranteed by the reference signal continue to hold. However, for satisfactory performance of the closed-loop system, uncertainties need to be accounted for: both within feedforward and feedback blocks. Depending on the nature of the uncertainties, techniques from robust and stochastic control can be utilized. The former deals with uncertainties which are known to belong to a certain class or a set, while the latter handles uncertainties that are known to possess a probability distribution \cite{mayne2015robust}. However, in many real-world circumstances (e.g., search and rescue post-natural disaster and planetary landing on unknown terrain), it is not possible to quantify or model all sources of uncertainties and contingencies in the operational environment a priori. 

So, how do we reduce the likelihood of mission failure or off-nominal performance in such cases? The proposed work addresses this question in the context of constrained trajectory generation for a vehicle with a known initial state. One approach to staying immune to unmodeled uncertainties and contingencies (i.e., ``unknown'' unknowns) is considering multiple candidate targets, instead of a single target. We focus our attention on the feedforward block and ask the question: how can the reference signal help with the identification of the best target? In the real-world setting, it is reasonable to assume that the vehicle can acquire new actionable information as it moves towards the candidate targets. So, deferring the decision to commit to/select a target provides time to learn (acquire more information) about the uncertainties. In a closed-loop system, for example, the perception and measurement updates provide new information. Ensuring that the candidate targets are reachable for as long as possible enables the vehicle to acquire new information for as long as possible, which then informs the best choice of target, where ``best'' could mean safest or cost-efficient.

The example of soft landing of a spacecraft on an unknown planetary terrain highlights the usefulness of deferring decision. The goal is to land a spacecraft at one of landing sites from a collection of candidates (see Figure \ref{fig:motivate}). We only possess coarse-grained information about the terrain and need to wait until the spacecraft nears the terrain to gather high-resolution information. It is straightforward to imagine a situation where the spacecraft chooses a landing site prematurely, but determines that it is not viable after getting too close, at which point recovery or divert might be impossible. Therefore, deferring decision is useful in this scenario for keeping a collection of landing sites reachable for a period of time while high-resolution terrain information is gathered to determine which candidate landing sites are viable \cite{brugarolas2017guidance}. Note that interplanetary communication has large latency which makes it infeasible to control the spacecraft in real-time with an Earth-based operator.
\begin{figure}[!ht]
    \centering
    \input{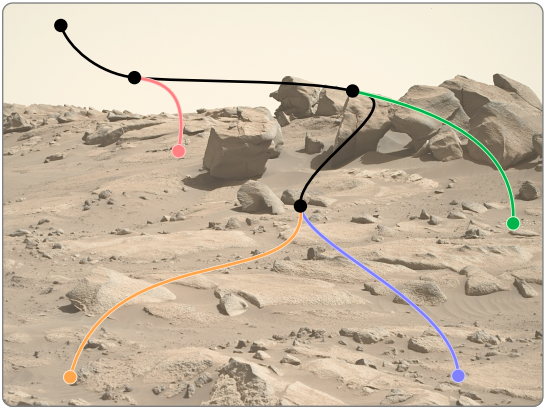}
    \caption{A Mars landing example where deferring decision is useful. The black trajectory segments keep a collection of candidate landing sites reachable (colored nodes). Each black node serves as a decision point beyond which reachability to one of the landing sites is lost. While the spacecraft follows the black segment, it can learn more about the terrain to determine the most viable landing site. The background image (taken by the Perseverance rover in December 2023 \cite{mars2020week149}) shows examples of a priori unknown irregularities on the Martian surface which can potentially make landing sites infeasible.}
    \label{fig:motivate}
\end{figure}

We present {\ddto}---deferred-decision trajectory optimization, a fully deterministic framework for formulating constrained trajectory generation in the presence of such unmodeled uncertainties and contingencies. First, we propose constrained-reachability-based formulations for {\ddto}---not with intention of directly solving the resulting optimization, but with the goal of discovering/analyzing the solution structure. Constrained reachable sets offer the most natural way to describe the notion of deferring decision. While they are intractable to compute for nonlinear systems with state dimension greater than four \cite[Fig. 2]{chen2018hamilton}, they are still useful building blocks for analyzing the problem at hand. Next, we show that the proposed constrained-reachability formulations have equivalent representations as cardinality minimization problems. While the cardinality minimization problems are also intractable to solve in general, their purpose here is to illuminate the structure of an optimal solution. With this knowledge we can design approximate, effective and efficient solution methods that are specialized to trajectory optimization problems. Note that we do not explore the well-known, general-purpose solution methods for general cardinality minimization problems, such as compressed sensing \cite{donoho2006compressed} and sparse recovery methods \cite{marques2019review}. 

We propose three solution methods that either solve the cardinality minimization problems or ensure the corresponding solution structure. Among these, two methods are designed for discrete-time affine systems subject to convex constraints, where we leverage mixed-integer conic programming (MICP) and quasiconvex optimization. The third method handles continuous-time nonlinear systems subject to nonconvex constraints, where we use the recently developed sequential convex programming (SCP) framework for nonconvex trajectory generation \cite{elango2024successive}. The method based on quasiconvex optimization was earlier proposed as a heuristic approach for {\ddto} in \cite{elango2022deferring}, which was followed by another variant, \textsc{adaptive}-{\ddto}, demonstrated within a closed-loop simulation framework \cite{haynerbuckner2023halo} designed for precision landing on an unknown planetary terrain with real-time vision-based perception.

We infer that the optimal solution for {\ddto} has a tree-like structure with trunk and branch trajectories (see Figure \ref{fig:motivate}). Hence, graph-based techniques such as sampling-based planners \cite{orthey2023sampling} might seem naturally suited for modeling and computing such structures. However, we adopt optimization-based modeling to place emphasis on constraint satisfaction and dynamic feasibility, which allows us to harness a variety of specialized techniques such as multiple-shooting and time-dilation \cite{elango2024successive}, and state-of-the-art convex optimization-based solution methods which are efficient and scalable \cite{domahidi2013ecos,yu2023extrapolated}. Furthermore, it is a natural concern that deferring decision (which is essentially procrastination) could come at the expense of some trajectory cost---the vehicle must not consume all of the onboard power resources or fuel in its pursuit of ``buying more time''. Therefore, in the proposed formulation and solution methods, we explicitly call out the importance of imposing constraints on the cumulative cost of a trajectory.

\subsection{Related Work}
Ideas related to deferred decision-making have received attention in the robotics and motion planning literature. An approach for air traffic congestion resolution which uses deferred decisions \cite{zobell2009deferability} demonstrates that conserving future flexibility can help manage the risk of uncertain predictions. Deferred decisions can counter the accumulation of errors in INS-based navigation \cite{bruder1999terrain} by  simultaneously allowing multiple potential trajectories to exist and selecting the best one to represent the robot's path. 

Enhancements to model predictive control (MPC) to anticipate future uncertainties have been achieved. A feedback min-max MPC approach was introduced in \cite{scokaert1998minmax}, where a family of control sequences is optimized at each time, each one corresponding to a different disturbance profile. Consequently, this approach avoids the likely feasibility problems that result from the use of other min-max formulations that optimize a single control profile over all possible future disturbance realizations. However, the knowledge of the type and source of the uncertainties is still necessary for this approach. The optimal cost design for MPC in \cite{jain2021optimal} demonstrates that delaying the decision until later implicitly accounts for the fact that the agent will get more information in the future and be able to make a better decision, which is possible by purposefully choosing a new MPC cost that is different from the planned trajectory cost. The branching MPC framework in \cite{chen2022interactive} considers a finite set of policies to represent the continuous spectrum of reactive behaviors of an uncontrolled agent. This approach bears resemblance to the notion of interaction-awareness in \cite{wang2023interaction}. The anomaly-detection framework proposed in \cite{sinha2024realtime} leverages the zero-shot reasoning capabilities of large language models (LLM) along with MPC to maintain a collection of safe trajectories that can handle out-of-distribution failure scenarios.

Methods developed for contingency planning and abort guidance are also related to deferred decisions since they accomplish a similar goal. An approach for computing abort trajectories along with the nominal is provided in \cite{lu2021abort}. The vehicle can divert to a predetermined safe region using a precomputed abort trajectory if any off-nominal effects are detected. Similar contingency planning measures are proposed for autonomous driving applications in \cite{hardy2013contingency,alsterda2019contingency}. 

Besides the robotics applications, deferred decision-making has been incorporated in floorplanning algorithms for very large scale integration (VLSI) \cite{yan2008defer}, data classification tasks \cite{baram1998partial}, and active fault diagnosis \cite{puncochar2018multiple}.

Finally, research from psychology shows that deferring decision is a likely outcome when there are many options to choose from \cite{tversky1992choice}, i.e., contrary to the principle of value maximization, choices can introduce conflict. Moreover, studies from organizational psychology have determined that procrastination plays a key role in improving human decision-making \cite{grant2016surprising}. The ability to defer the decision about choosing a goal can improve the performance and reliability of human-in-the-loop (HIL) robotic systems.

\subsection{Notation}
The set of real numbers is denoted by $\bR$, the set of nonnegative real numbers by $\bR_+$, the set of $n$-dimensional real vectors by $\bR^n$, and the set of $m\times n$ matrices by $\bR^{m\times n}$. The set of integers between (and including) $a$ and $b$ (with $a \le b$) is denoted by $\bZr{a}{b}$, and the cardinality of a set $C$ by $|C|$. The vector of zeros in $\bR^n$ is denoted by $\zeros{n}$, and the identity matrix in $\bR^{n\times n}$ by $\eye{n}$. The concatenation of vectors $v\in\bR^n$ and $w\in\bR^m$ is denoted by $(v,w)\in\bR^{n+m}$. The function $x \mapsto \znrm{x}$ indicates whether a vector $x$ is non-zero, i.e.,
\[
    \znrm{x} \triangleq \left\{ \begin{array}{ll}
         1 &\text{if } x \ne 0,  \\
         0 &\text{otherwise}.
    \end{array} \right.
\]
\section{Preliminaries}
Consider a time-invariant, discrete-time nonlinear dynamical system 
\begin{align}
    x_{k+1} = f(x_k,u_k),\quad k \ge 1,\label{eq:dyn-sys}
\end{align}
where index $k$ corresponds to time, $x_k\in\bRnx$ is the state, and $u_k\in\bRnu$ is the control input. The state and control input are required to lie in sets $\bX\subseteq\bRnx$ and $\bU\subseteq\bRnu$, respectively. We refer to a sequence of $N$ states as a trajectory of horizon length $N$. A trajectory is said to be feasible if the sequence of states lie in $\bX$, the corresponding sequence of control inputs lie in $\bU$, and \eqref{eq:dyn-sys} holds for the trajectory and control input sequence. We are interested in a feasible trajectory of horizon length $N$ that starts from $z^0\in\bX$ and terminates at one of the $n$ targets represented by sets $\mc{Z}^1,\ldots,\mc{Z}^n\subseteq \bX$.

The space of feasible trajectories and control input sequences for \eqref{eq:dyn-sys} can be characterized by constrained forward and backward reachable sets. The set of all terminal states of the feasible trajectories of horizon length $M+1$ with initial state $z\in\bX$ is denoted by $\cF_{M}(z)$.
\begin{definition} Given $M \ge 1$, the $M$-step constrained forward reachable set of $z\in\bX$ is given by
\[
    \cF_{M}(z) \triangleq \left\{\,x_{M+1}\,\middle|\,\begin{array}{l} u_1,\ldots,u_M\in\bU\\
    x_{k+1}=f(x_k,u_k),~k\in\bZr{1}{M}\\
    x_1=z,~x_2,\ldots,x_{M+1}\in\bX
    \end{array}\!\right\}.
\]
Furthermore, $\cF_0(z) \triangleq \{z\}$.\label{def:k-CFRS}
\end{definition}
The set of all initial states of the feasible trajectories of horizon length $M+1$ that terminate at $\mc{Z}\subseteq\bX$ is denoted by $\cB_M(\mc{Z})$.
\begin{definition} Given $M\ge 1$, the $M$-step constrained backward reachable set of $\mc{Z}\subseteq\bX$ is given by
\[
    \cB_M(\mc{Z}) \triangleq \left\{\,x_1\,\middle|\,\begin{array}{l} u_1,\ldots,u_M\in\bU\\
    x_{k+1}=f(x_k,u_k),~k\in\bZr{1}{M}\\
    x_1,\ldots,x_M\in\bX,~x_{M+1} \in \mc{Z}
    \end{array}\!\right\}.
\]
Furthermore, $\cB_0(\mc{Z}) \triangleq \mc{Z}$.\label{def:k-CBRS}
\end{definition}
The above definitions for the constrained forward and backward reachable sets are closely related to Definitions 10.6 and 10.5, respectively, in \cite{borrelli2017predictive}. We have the following result about the states on a trajectory of horizon length $N$ which are recursively selected from the constrained reachable sets.
\begin{lemma} Consider a trajectory $x_1,\ldots,x_N$. Then, the following statements hold:
\begin{enumerate}
    \item If $x_1 = z^0$ and  $x_k\in\cF_1(x_{k-1})$ for $k\in\bZr{2}{N}$, then $\cF_1(x_{k-1}) \subseteq \cF_{k-1}(z^0)$ for $k\in\bZr{2}{N}$.
    \item If $x_N\in\mc{Z}^j$ for some $j\in\bZr{1}{n}$ and 
    $x_{k-1}\in\cB_1(\{x_k\})$ for $k\in\bZr{2}{N}$, then $\cB_1(\{x_k\})\subseteq\cB_{N-k+1}(\mc{Z}^j)$ for $k\in\bZr{2}{N}$.
\end{enumerate}\label{lem:1-step-to-k-step-RS}
\end{lemma}
\begin{proof}
1. Clearly, $x_1, \dots, x_N \in \bX$. The inclusion holds trivially for $k = 2$. We therefore fix $k\in\bZr{3}{N}$. Let $y \in  \cF_1(x_{k-1}) $. Then, $y \in \bX$ and there exists a $u_{k-1} \in \bU$ such that $y = f(x_{k-1}, u_{k-1})$. 
By the assumption, there exist $u_1, \ldots, u_{k-2} \in \bU$ such that $x_i = f(x_{i-1}, u_{i-1})$ for $i  \in\bZr{2}{k-1} $. We can now conclude that $y \in \cF_{k-1}(z^0) $.\\
2. Clearly, $x_1, \dots, x_N \in \bX$. The inclusion holds trivially for $k = N$. We therefore fix  $k\in\bZr{2}{N-1}$. Let $y \in  \cB_1(\{x_k\})$.  Then, $y \in \bX$ and there exists a $u_{k-1} \in \bU$ such that $x_k = f(y, u_{k-1})$.  By the assumption, there exist $u_k, \ldots, u_{N-1} \in \bU$ such that  $x_i = f(x_{i-1}, u_{i-1})$ for $i  \in\bZr{k+1}{N} $. We can now conclude that $y \in \cB_{N-k+1}(\{x_N\})$.
\end{proof}
\begin{remark}\label{rem:Fxtm1-notsubset-Ftz0}Let $x_1, \ldots, x_N$ be a trajectory satisfying $x_k\in\cF_{k-1}(z^0)$ for $k\in\bZr{1}{N}$. Is it true that such a trajectory is dynamically feasible, i.e., $x_{k}\in\cF_1(x_{k-1})$ for $k\in\bZr{2}{N}$? No, it is not true in general, as illustrated by the following example. Consider the dynamical system: $x_{k+1} = x_k + (u_k,u_k^2)$, with $x_k\in\bR^{2}$ and $u_k\in\bR$. The system state is subject to the constraint $a^\top x_k \le 9$ for all $k\ge 1$, where $a = (0,1)$, and the initial state is $z^0 \triangleq \zeros{2}$. Observe that $x_2\triangleq(3,9)\in\mc{F}_1(z^0)$ and $x_3 \triangleq (4,8)\in\cF_{2}(z^0)$. However, $x_3\notin\cF_1(x_2)$.
\end{remark}
Given any $J\subseteq\bZr{1}{n}$ and $k\in\bZr{1}{N}$, the intersection of the $(N-k)$-step constrained backward reachable sets of targets $\mc{Z}^j$, for $j\in J$, forms a key building block for assessing reachability to targets from a given trajectory. 
\begin{definition} Given nonempty $J\subseteq\bZr{1}{n}$ and $k\in\bZr{1}{N}$, define
\[
    \cB^J_{N-k} \triangleq \displaystyle \underset{j\in J}{\medcap} \cB_{N-k}(\mc{Z}^j).
\]\label{def:Nmt-BRS-J}
\end{definition}\vspace{-0.5cm}
For brevity in the subsequent discussion, we say ``targets in $J$'' instead of ``targets $\mc{Z}^j$, for $j\in J$'', for any $J\subseteq\bZr{1}{n}$. Another key construction associated with each target is the $k$-reach set. It characterizes the set of states for which each is the $k$th state of a feasible trajectory of horizon length $N$ to the target.
\begin{definition}[$\bm{k}$-reach set] Given $j\in\bZr{1}{n}$ and $k\in\bZr{1}{N}$, define
\[
    \cR^j_k \triangleq \cF_{k-1}(z^0) \medcap \cB_{N-k}(\mc{Z}^j).
\]\label{def:k-reach}  
\end{definition}
It is immediate from Definition \ref{def:k-reach} that every point in the $k$-reach set is the $k$th state of a feasible trajectory of horizon length $N$ from $z^0$ to $\mc{Z}^j$. The following result is the converse of this statement.
\begin{figure}[!ht]
\centering



\begin{tikzpicture}[scale=0.9]

    \def\radius{2pt}
    \def\setradius{3pt}

    \node[rectangle, fill=beige!20, draw=black!50, minimum width=8.2cm, minimum height=4.4cm, rounded corners, line width=1pt] at (3.9,-0.15) {};

    \coordinate (p1) at (0,    0.325);
    \coordinate (p2) at (0.75, 0.4);
    \coordinate (p3) at (1.85, 0.5);
    \coordinate (p4) at (3.5,  0.5);
    \coordinate (p5) at (4,    -0.1);
    \coordinate (p6) at (4.24,0);
    \coordinate (p7) at (5.85,0.2);
    \coordinate (p8) at (7,   0.3);
    \coordinate (p9) at (7.75,0.55);

    \fill[blue!20]                       (0,   0.325) circle (\setradius);
    \node[below,yshift=-0.05cm] at       (0,   0.325) {$\mathcal{F}_0$};  
    \fill[black]                         (0,   0.325) circle (\radius) node[above, yshift=0.05cm] {$z^0$};         
    \filldraw[fill=blue!30, draw=none]   (0.75,0) ellipse (0.3 and 0.6) node {$\mathcal{F}_1$};     
    \filldraw[fill=blue!40, draw=none]   (1.85,0) ellipse (0.5 and 0.9)  node {$\mathcal{F}_2$};                 
    \filldraw[fill=blue!50, draw=none]   (3.5, 0) ellipse (0.75 and 1.2) node[above] {$\mathcal{F}_3$};
    
    \fill[red!20]                                    (7.75,-0.683) circle (\setradius);
    \node[below,yshift=-0.09cm,xshift=0.04cm] at                   (7.75,-0.683) {$\mathcal{B}_0$};     
    \fill[black]                                     (7.75,-0.683) circle (\radius) node[above, yshift=0.09cm, xshift=0.04cm] {$\mc{Z}^j$};       
    \filldraw[fill=red, draw=none, fill opacity=0.3] (7,   -1) ellipse (0.3 and 0.6); 
    \node at                                         (7,   -1) {$\mathcal{B}_1$};     
    \filldraw[fill=red, draw=none, fill opacity=0.4] (5.85,-1) ellipse (0.5 and 0.9);
    \node at                                         (5.85,-1) {$\mathcal{B}_2$};                 
    \filldraw[fill=red, draw=none, fill opacity=0.5] (4.25,-1) ellipse (0.75 and 1.2);
    \node at                                         (4.25,-1) {$\mathcal{B}_3$};

    \fill[black]                                     (7.75,0.55) circle (\radius) node[above, yshift=0.04cm, xshift=0.04cm] {$\mc{Z}^i$};

    \def\ypos{1.5}

    \draw[-{To[length=1.5mm,width=2.5mm]}, thick, line width=1pt] (-0.25, \ypos) -- (8.1, \ypos) node[above] {$k$};

    \foreach \x in {0,0.75,1.85,4,5.85,7,7.75} {
        \draw[thick, line width=1pt] (\x,\ypos-0.1) -- (\x,\ypos+0.1) node {};    
    }   
    
    \draw[thick, black] (p1) to[curve through={(p2) (p3)  (p5) (p6) (p7) (p8)}] (p9);     
    
    \foreach \i in {2,3,5,7,8}
        \fill[black] (p\i) circle (\radius);    

    \node[anchor=north, xshift=-0.05cm] at (p5) {$x_4$};

\end{tikzpicture}

%
\caption{A trajectory of horizon length $7$ from $z^0$ to $\mc{Z}^i$ passing through $\cR^j_4$ at the fourth time node. The arguments of $\cF_{k}$ and $\cB_{k}$, for $k\in\bZr{0}{3}$, are omitted for brevity.}
\label{fig:mutual-reach-illustrate}
\end{figure}
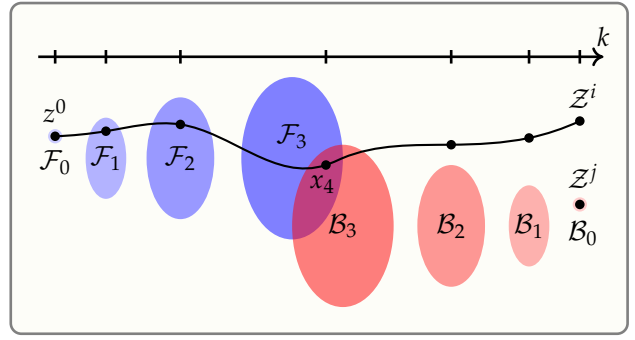
\begin{lemma}\label{lem:converse-for-k-reach}
Any feasible trajectory of horizon length $N$ from $z^0$ to $\mc{Z}^j$ intersects $\cR^j_k$ at time $k$ for $k \in \bZr{1}{N}$.
\end{lemma}
\begin{proof}
Let $x_1, \ldots, x_{N}$, with $x_1=z^0$ and $x_N\in\mc{Z}^j$, be a feasible trajectory of horizon length $N$. Then, $x_k\in\cF_{k-1}(z^0)$ and $x_k\in\cB_{N-k}(\mc{Z}^j)$, for $k\in\bZr{1}{N}$, from Definitions \ref{def:k-CFRS} and \ref{def:k-CBRS}, respectively. 
\end{proof}
Next, to quantify reachability to multiple targets, we consider the intersection of $k$-reach sets of multiple targets. The scenario in Figure \ref{fig:mutual-reach-illustrate} shows the existence of a  feasible trajectory of horizon length $4$ from $x_4$ to target $j$, where $x_4$ is the fourth state on a feasible trajectory to target $i$. In other words, $x_4\in\cR_4^j\cap \cR^i_4$. We use Definitions \ref{def:Nmt-BRS-J} and \ref{def:k-reach} for representing the intersection of $k$-reach sets of targets in $J$.
\begin{definition} Given  nonempty $J\subseteq\bZr{1}{n}$ and $k\in\bZr{1}{N}$, the intersection of $k$-reach sets of targets in $J$ is denoted by
\[
    \cR^J_k \triangleq \underset{j\in J}{\medcap} \cR^j_k =  \cF_{k-1}(z^0)\medcap \cB^J_{N-k}.
\]\label{def:k-reach-J}
\end{definition}\vspace{-0.4cm}
\begin{definition}Given $k\in\bZr{1}{N}$, the collection of all subsets of targets indices (i.e., subsets of $\bZr{1}{n}$), for which the intersection of $k$-reach sets is nonempty, is denoted by
\[
    \Lambda_k \triangleq \{\,J \subseteq \bZr{1}{n}\,|\,\cR^J_{k} \ne \emptyset, J\ne\emptyset\,\}.
\]\label{eq:Lamk}
\end{definition}\vspace{-2em}
{\ddto} seeks trajectories, from initial state to targets, that maximize the size of the element of $\Lambda_k$ selected for each $k\in\bZr{1}{N}$, in a cumulative sense.

The following lemma describes the consequence of the $k$-reach set of a target collection being empty and nonempty at a particular time instant.
\begin{lemma}For any nonempty $J\subseteq \bZr{1}{n}$, the following holds:
\begin{enumerate}
    \item If there exists a $\ol{k}\in\bZr{1}{N}$ such that $\cR^J_{\ol{k}} \ne \emptyset$, then $\cR^J_k  \ne \emptyset$ for each $k \in\bZr{1}{\ol{k}}$.
    \item If there exists a $\ul{k}\in\bZr{1}{N}$ such that $\cR^J_{\ul{k}} = \emptyset$, then $\cR^J_k = \emptyset$ for each $k \in\bZr{\ul{k}}{N}$.
\end{enumerate}\label{lem:interset-k-reach}
\end{lemma}
\begin{proof}1. If $\ol{k} = 1$ such that $\cR^J_{\ol{k}} \ne \emptyset$, then the statement is trivially true. We therefore assume that $\ol{k}  \in\bZr{2}{N}$ such that $\cR^J_{\ol{k}} \ne \emptyset$. 
Let $\ol{z} \in  \cR^J_{\ol{k}} $. Then, $\ol{z} \in \cF_{\ol{k}-1}(z^0) $ and $\ol{z} \in  \cB_{N-\ol{k}}(\mc{Z}^j) $  for all $j \in J$. 
If $\ol{k} =  N$, this tells us that there is a sequence $x_1, \ldots, x_{\ol{k}}$ with $x_1 = z^0 $ and $x_{\ol{k}} = \ol{z}$ such that $x_1, \ldots, x_{\ol{k}}$ is a feasible trajectory from $z^0$ to $\mc{Z}^j$ for all $j \in J$. In conjunction with Lemma \ref{lem:converse-for-k-reach}, this implies that $x_k  \in\cR^J_k $ for each $k\in\bZr{1}{\ol{k}}$.
On the other hand, if $\ol{k} \neq  N$, this tells us that there exist not only a sequence $x_1, \ldots, x_{\ol{k}}$ with $x_1 = z^0 $ and $x_{\ol{k}} = \ol{z}$ but also for each $j \in J$, a sequence $x_{\ol{k}+1}^j,\ldots,x^j_N$ such that $x_1, \ldots, x_{\ol{k}}, x_{\ol{k}+1}^j,\ldots,x^j_N $ is a feasible trajectory from $z^0$ to $\mc{Z}^j$. Hence, we have that $x_k  \in\cR^J_k $ for each $k\in\bZr{1}{\ol{k}}$. 

2. If $\ul{k} = N$ such that $\cR^J_{\ul{k}} = \emptyset$, then the statement is trivially true. We therefore let $\ul{k}  \in\bZr{1}{N-1}$ such that $\cR^J_{\ul{k}} = \emptyset$. Assume, on the contrary, that there is a $k' \in  \bZr{\ul{k}+1}{N}$ such that 
$\cR^J_{k^\prime} \neq \emptyset$. By the first part of the lemma, $\cR^J_{\ul{k}} \neq \emptyset$, which contradicts $\cR^J_{\ul{k}} = \emptyset$.
\end{proof}
It is natural that the goal of deferring decision (or procrastination) can be at odds with the goal of managing the onboard power and fuel resource: long-duration maneuvers that keep multiple targets reachable will consume more fuel. Hence, it is necessary to impose constraints based on the cumulative cost of a trajectory. 
\begin{remark}\label{rem:aug-sys-cost}The definitions of the constrained forward and backward reachable sets only consider pointwise (in time) constraints on the state and control input. Constraints based on the cumulative value of a function of state or control input evaluated for the entire trajectory can also be considered in Definitions \ref{def:k-CFRS} and \ref{def:k-CBRS}, by appending an additional state to the system dynamics. This approach turns a cumulative constraint on the trajectory and control input sequence into an equivalent terminal constraint on the additional state. For instance, the following constraint 
\begin{align}
    \sum_{k=1}^{N-1} l(x_k,u_k) \le l_{\max},\label{eq:cum-cnstr}
\end{align}
where $l$ is a stage cost function and $l_{\max}\in\bR$ is an upper bound, is transformed by augmenting the original system \eqref{eq:dyn-sys} as follows 
\begin{align}
\tilde{x}_{k+1} = \begin{bmatrix}
    x_{k+1} \\[0.1cm] \theta_{k+1}
    \end{bmatrix} = \tilde{f}(\tilde{x}_k,u_k) \triangleq  \begin{bmatrix}
                    f(x_k,u_k) \\[0.1cm] \theta_k + l(x_k,u_k)
                    \end{bmatrix}.\label{eq:aug-sys}    
\end{align}
Suppose $\mc{Z}$ is the original target. By requiring an augmented state trajectory $\tilde{x}_1,\ldots,\tilde{x}_N$ to terminate at the augmented target $\tilde{\mc{Z}}\triangleq \mc{Z}\times\{\theta\in\bR\,|\,\theta\le l_{\max}\}$ with initial condition $\theta_1=0$, we can satisfy \eqref{eq:cum-cnstr}.
\end{remark}
The goal of {\ddto}, roughly speaking, is to generate trajectories that ensure reachability to as many of the targets for as long as possible. To this end, we explore two different modeling approaches. First, we optimize a trajectory that remains within the intersection of the $k$-reach sets of a (time-varying) collection of targets indices. Second, we jointly optimize trajectories to each of the targets such that, for a collection of targets at each time, the states on the corresponding trajectories are identical. We demonstrate the equivalence between the former approach (which relies on set-valued decision variables) and the latter approach (where only the trajectories and control input sequences are decision variables).

We require the following assumption for the remainder of the discussion.
\begin{assumption}
The dynamical system \eqref{eq:dyn-sys}, initial state $z^0$, targets: $\mc{Z}^j$, for $j\in\bZr{1}{n}$, and constraint sets $\bX$ and $\bU$, are chosen such that: 
\begin{itemize}
    \item $\mc{R}_N^j \ne \emptyset$, for each $j\in\bZr{1}{n}$. 
    \item $\cR_1^{\scriptstyle[1:n]}\ne \emptyset$.
\end{itemize}\label{assumption}
\end{assumption}
\begin{remark}
Including the time-dilation approach, described in Appendix \ref{app:dilation}, within the proposed framework allows different final times for trajectories to each target, despite them having the same horizon length of $N$.\label{rem:time-dilation}
\end{remark}
\section{Constrained Reachability-Based DDTO}\label{sec:reach}
We seek a trajectory of horizon length $N$ to one of the targets, and for each $k\in\bZr{1}{N}$, we consider $(N-k)$-step constrained reachability to other targets from the state at time $k$ on the trajectory. 
\subsection{Maximize duration of reachability to a collection of targets}
\begin{definition}[Branch time] Given nonempty $J\subseteq\bZr{1}{n}$, branch time $k^J$ is the latest time for which targets in $J$ are reachable in $N-k^J$ steps from a point in $\cR^J_k$, i.e.,
\[
    k^J \triangleq \max \{\,k\in\bZr{1}{N}\,|\,\cR^J_k\ne\emptyset\,\}.
\]\label{def:kJ}
\end{definition}\vspace{-0.5cm}
Since $\cR^J_k \ne \emptyset$ and $J \neq \emptyset$ iff $J\in\Lambda_k$, the branch time can be equivalently written as $k^J \triangleq \max\{\,k\in\bZr{1}{N}\,|\,J\in\Lambda_k\}$.
The following proposition is now immediate.
\begin{proposition}\label{prop:traj-cant-intersect-after-kJ}
  For each $j \in  J$, let 
   $x_1^j, \ldots, x_N^j$ be a feasible trajectory from $z^0$ to $\mc{Z}^j$.
 Then, these trajectories do not simultaneously intersect at any time later than the branch time. 
\end{proposition}
\subsection{Maximize reachable targets from trajectory to a particular target} 
Given $i\in\bZr{1}{n}$, a solution to \eqref{prb:opt-single-traj-i} consists of a trajectory to target $i$ which maximizes, in a cumulative sense, the number of reachable targets at each time instant, i.e., the trajectory selects the largest possible member of $\Lambda_k$, for each $k\in\bZr{1}{N-1}$.
\begin{subequations}
\begin{align}
\underset{x_k,\,J_k}{\mr{maximize}}~&~\sum_{k=1}^{N} |J_{k}|, & &\label{eq:opt-single-traj-obj}\\
\mr{subject~to}~&~x_k \in \cF_1(x_{k-1})\medcap \cB^{J_k}_{N-k}, & & k \in \bZr{2}{N},\label{eq:opt-dyn-feas-i}\\
 &~J_k\in\Lambda_k, & & k\in \bZr{1}{N},\\
 &~x_1 = z^0,~J_N \supseteq \{i\}.\label{eq:opt-pick-targ-i}
\end{align}\label{prb:opt-single-traj-i}%
\end{subequations}
Note that the construction of \eqref{eq:opt-dyn-feas-i} utilizes Lemma \ref{lem:1-step-to-k-step-RS} and Definition \ref{def:k-reach-J}. 
\subsection{Maximize reachable targets from trajectory to an arbitrary target}
A solution to \eqref{prb:opt-single-traj} consists of a trajectory to a target in $\bZr{1}{n}$ which maximizes, in a cumulative sense, the number of reachable targets at each time instant, i.e., the trajectory will terminate at a target that allows it to select the largest possible member of $\Lambda_k$, for each $k\in\bZr{1}{N-1}$.
\begin{subequations}
\begin{align}
\underset{x_k,\,J_k}{\mr{maximize}}~&~\sum_{k=1}^{N} |J_{k}|, & & \label{eq:opt-cost}\\
\mr{subject~to}~&~x_k \in \cF_1(x_{k-1})\medcap \cB^{J_k}_{N-k}, & & k \in \bZr{2}{N},\label{eq:opt-dyn-feas}\\
 &~J_k\in\Lambda_k, & & k\in \bZr{1}{N},\label{eq:opt-single-traj-Jk}\\
 &~x_1 = z^0.\label{eq:opt-init-cond}
\end{align}\label{prb:opt-single-traj}%
\end{subequations}
\begin{remark}
We obtain \eqref{prb:opt-single-traj} from \eqref{prb:opt-single-traj-i} by removing the explicit boundary condition constraint on $J_N$ in \eqref{eq:opt-pick-targ-i}, i.e., the choice of target is unspecified in \eqref{prb:opt-single-traj}. Constraints \eqref{eq:opt-dyn-feas} and \eqref{eq:opt-single-traj-Jk} are sufficient to ensure that the trajectory terminates at one of the targets.    
\end{remark}
\begin{lemma}[Monotonicity]\label{lem:monotonicity} The sequence of sets $J_1,\ldots,J_N$, which forms a solution to \eqref{prb:opt-single-traj}, must satisfy $J_k \subseteq J_{k-1}$ for $k\in\bZr{2}{N}$.
\end{lemma}
\begin{proof}
Let $x_k,J_k$, for $k\in\bZr{1}{N}$, be a solution to \eqref{prb:opt-single-traj}. Suppose that $J_k\setminus J_{k-1}\ne\emptyset$ for some $k\in\bZr{2}{N}$. Pick $j\in J_k\setminus J_{k-1}$. Then, due to Lemma \ref{lem:1-step-to-k-step-RS} and Definition \ref{def:k-reach-J}, \eqref{eq:opt-dyn-feas} implies that $x_k\in\cR^j_k$ . If $k=N$ then is clear that $x_{k-1}$ lies on a feasible trajectory to $\mc{Z}^j$, i.e., $x_{k-1}\in\cR^j_{k-1}$ from Lemma \ref{lem:converse-for-k-reach}. If $k\in\bZr{2}{N-1}$ there exists a sequence $x^j_{k+1},\ldots,x^j_N$ such that $x_1,\ldots,x_{k-1},x_k,x^j_{k+1},\ldots,x^j_N$ is feasible trajectory from $z^0$ to $\mc{Z}^j$. Then $x_{k-1}\in\cR^j_{k-1}$ from Lemma \ref{lem:converse-for-k-reach}. Therefore, $J_{k-1}$ can be enlarged to include $j$. This contradicts the assumption that $J_{k-1}$ forms a solution to \eqref{prb:opt-single-traj} (i.e., objective function \eqref{eq:opt-cost} is maximized).
\end{proof}
The monotonicity property described in Lemma \ref{lem:monotonicity} is also satisfied by solutions of \eqref{prb:opt-single-traj-i}. Further, a sequence of sets $J_1,\ldots,J_N$, which forms a solution to \eqref{prb:opt-single-traj-i}, satisfies $i\in J_k$ for $k\in\bZr{1}{N}$.
\begin{definition}[Branch time/point]\label{def:branch-kj}
Let $x_k,J_k$, for $k\in\bZr{1}{N}$, be a solution to either \eqref{prb:opt-single-traj-i} or \eqref{prb:opt-single-traj}. The latest time target $j\in\bZr{1}{n}$ is reachable is a branch time, given by
\begin{align}
    k^j \triangleq{} & \max \{\,k\,|\,j \in J_k\,\},\label{eq:tj-opt-single-traj}
\end{align} 
and the corresponding state, $x_{k^j}$, is a called a branch point.
\end{definition}
\begin{corollary} \label{cor:construct-all-traj-from-single} Let $x_k,J_k$, for $k\in\bZr{1}{N}$, be a solution to \eqref{prb:opt-single-traj-i} or \eqref{prb:opt-single-traj}, with branch times given by \eqref{eq:tj-opt-single-traj}. Since $x_{k^j}\in\cR^j_{k^j}$, for each $j\in\bZr{1}{n}$, there exists a feasible trajectory of horizon length $N-k^j+1$ from $x_{k^j}$ to $z^j$, denoted by $x^j_{k^j},\ldots,x^j_N$. Also, if $k^j>1$, let $x^j_k \triangleq x_k$, for $k\in\bZr{1}{k^j-1}$. Then $x^j_1,\ldots,x^j_N$ is a feasible trajectory to target $j$.
\end{corollary}
The trajectories constructed in Corollary \ref{cor:construct-all-traj-from-single} are optimal in the sense of cardinality minimization, which will be discussed in the subsequent section.
\begin{remark} Constraints \eqref{eq:opt-dyn-feas-i} and \eqref{eq:opt-dyn-feas}, which ensure dynamic feasibility, cannot alternately be represented using $1$-step constrained backward reachable sets. More precisely, a constraint such as
\[
    x_{k} \in \cB_1(\{x_{k+1}\}) \medcap \cF_{k-1}(z^0), \quad  k\in\bZr{1}{N-1},
\]
along with the boundary condition $x_N \in \mc{Z}^j$, for some $j\in\bZr{1}{n}$, will result in a na\"{i}ve single-target trajectory optimization problem which disregards reachability to other targets.
\end{remark}
\begin{remark} Consider a trajectory $x_1,\ldots,x_N$. If the following implication 
\begin{align}
    x_k\in\cF_{k-1}(z^0) \implies x_k\in\cF_{1}(x_{k-1}), \label{eq:Fxtm1-notsubset-Ftz0} 
\end{align}
holds for each $k\in\bZr{2}{N}$, then \eqref{eq:opt-dyn-feas} is equivalent to the constraint $x_k\in \cR^{J_k}_k$, for $k\in\bZr{2}{N}$. In other words, $J_k$ should be chosen such that $\cR^{J_k}_k\ne\emptyset$, which is equivalent to constraint \eqref{eq:opt-single-traj-Jk}. Hence, \eqref{prb:opt-single-traj} simplifies to 
\begin{align}
\underset{J_k\in\Lambda_k}{\mr{maximize}}~&~\sum_{k=1}^{N} |J_{k}|,\label{prb:opt-Jk-only}
\end{align}
where the trajectory states are no longer explicit decision variables. A solution to \eqref{prb:opt-Jk-only} comprises of the largest cardinality member of $\Lambda_k$, for each $k\in\bZr{1}{N}$. However, Remark \ref{rem:Fxtm1-notsubset-Ftz0} shows that \eqref{eq:Fxtm1-notsubset-Ftz0} is not true in general. As a result, a feasible trajectory that selects the largest cardinality member of $\Lambda_k$, for each $k\in\bZr{1}{N}$, may not exist.
\end{remark}
\section{Cardinality Minimization-Based DDTO}\label{sec:card-min}
In this section we show that the constrained-reachability-based descriptions for {\ddto} can be equivalently stated as cardinality minimization problems where the set-based variables are eliminated.
\subsection{Maximize duration of reachability to a collection of targets}
The problem of maximizing the time duration for which targets in $J\subseteq\bZr{1}{n}$ are reachable is the same as the problem of generating trajectories to targets in $J$ which stay identical for as long as possible, which can be described as follows
\begin{subequations}
\begin{align}
\underset{x^j_k,\,u^j_k}{\mr{maximize}}~&~\gfunc^{|J|}(X^J), & & \label{eq:gJmax-qcvx-obj}\\
\mr{subject~to}~&~x^j_{k+1} = f(x^j_k,u^j_k), & & k\in\bZr{1}{N-1},\\
 &~x^j_k\in\bX,~u^j_k\in\bU, & & k\in\bZr{1}{N-1},\\
 &~x^j_1 = z^0,~x^j_N \in \mc{Z}^j, & & \\
 &~j\in J, & & \nonumber 
\end{align}\label{prb:gJmax}%
\end{subequations}
where $X^J = (\ldots,x^j_k,\ldots)$, with $k\in\bZr{1}{N}$ and $j\in J$, concatenates the states of all trajectories, and $\gfunc^m:\bR^{mn_xN}\to\bR$ evaluates the maximum time duration (starting from $k=1$) until which given $m$ trajectories stay identical. Then the objective function \eqref{eq:gJmax-qcvx-obj} is given by
\begin{align}
    & \gfunc^{|J|}(X^J) =\label{eq:gJ}\\
    & \max \big\{k^\prime\in\bZr{1}{N}\,\big|\,x^j_k=x^i_k~\forall\,k\in\bZr{1}{k^\prime},~\forall\,i,j\in J\big\}.\nonumber
\end{align}\vspace{-0.5cm}
\begin{theorem}The optimal value of \eqref{eq:gJmax-qcvx-obj} is the branch time $k^J$ from Definition \ref{def:kJ}.\label{lem:opt-gJmax}
\end{theorem}
\begin{proof} Let trajectories $\bar{x}^j_1,\ldots,\bar{x}^j_N$, for $j\in J$, form a solution to \eqref{prb:gJmax}, and let $k^\star$ be the optimal value of \eqref{eq:gJmax-qcvx-obj}. Suppose that $k^J < k^\star$. Then, from Definition \ref{def:kJ}, $\cR^J_{k^\star} = \emptyset$. Since $k^\star$ is the latest time instant until which the trajectories stay coincident, we have that $\bar{x}^i_{k^\star}\in\cR^j_{k^\star}$, for each $i,j\in J$. Let $z^\star$ denote the final state on the coincident portion of the trajectories. Then, $z^\star \in \cR_{k^\star}^j$, for $j\in J$. Hence, $z^\star \in \cR^J_{k^\star}$, which is a contradiction. 
    
Next, suppose that $k^\star < k^J$. Since $\cR^{J}_{k^J} \ne \emptyset$, pick $z^\star \in \cR^J_{k^J}$. If $k^J = N$ there exists a sequence of states $x_1,\ldots,x_{k^J-1}$ such that $x_1, \ldots, x_{k^J-1} , z^\star$ is a feasible trajectory to $\mc{Z}^j$, for all $j\in J$. If $k^J < N$ there exist sequences $x_1,\ldots,x_{k^J-1}$ and $x^j_{k^J+1},\ldots,x^j_N$ such that $x_1, \ldots, x_{k^J-1} , z^\star, x^j_{k^J+1},\ldots,x^j_N$ is a feasible trajectory to $\mc{Z}^j$, for each $j\in J$. Note that the states of these trajectories are coincident until time $k^J$, i.e., the value of \eqref{eq:gJmax-qcvx-obj} is $k^J$. This contradicts the assumption that the optimal value of \eqref{eq:gJmax-qcvx-obj} ($k^\star$) is strictly smaller than $k^J$.
\end{proof}
The above result also validates Proposition \ref{prop:traj-cant-intersect-after-kJ}.
\subsection{Maximize reachable targets from trajectory to a particular target}
\begin{figure}[!ht]
    \centering



\begin{tikzpicture}[x=0.75pt,y=0.75pt,yscale=-1,xscale=1,scale=0.4]

\node[rectangle, fill=beige!20, draw=black!50, minimum width=6.5cm, minimum height=5.3cm, rounded corners, line width=1pt] at (315,610) {};    

\def\radius{5}
\def\verticalshift{350} 
\def\horizontalshift{-1} 

\begin{scope}[yshift=\verticalshift,xshift=\horizontalshift] 

\draw [color=\customblue  ,draw opacity=1 ]   (73,261) .. controls (84.39,252.46) and (110.79,285.76) .. (164.75,297.75) .. controls (218.71,309.74) and (239,250) .. (344,274) .. controls (479,295) and (436.25,339.75) .. (560,348) ;
\draw [shift={(560,348)}, rotate = 3.81] [color=\customblue  ,draw opacity=1 ][fill=\customblue  ,fill opacity=1 ][line width=0.75]      (0, 0) circle [x radius= \radius, y radius= \radius]   ;
\draw [color=\customgreen  ,draw opacity=1 ]   (73,261) .. controls (79.75,253.25) and (106.75,275.75) .. (124.25,283.75) .. controls (184.25,301.25) and (173.75,230.25) .. (245.25,247.75) .. controls (329.75,274.75) and (338.94,256.63) .. (384.75,233.25) .. controls (420.97,214.77) and (453.76,177.96) .. (514.25,185.25) ;
\draw [shift={(514.25,185.25)}, rotate = 6.87] [color=\customgreen  ,draw opacity=1 ][fill=\customgreen  ,fill opacity=1 ][line width=0.75]      (0, 0) circle [x radius= \radius, y radius= \radius]   ;
\draw [color=\customred  ,draw opacity=1 ]   (73,261) .. controls (217,178) and (307,326) .. (530,270) ;
\draw [shift={(530,270)}, rotate = 345.9] [color=\customred  ,draw opacity=1 ][fill=\customred  ,fill opacity=1 ][line width=0.75]      (0, 0) circle [x radius= \radius, y radius= \radius]   ;

\draw[fill=black] (69.63,261) circle [x radius= \radius, y radius= \radius]   ;
\draw (27,235) node [anchor=north west][inner sep=0.75pt]    {$z^{0}$};
\draw (520,160) node [anchor=north west][inner sep=0.75pt]  [color=\customgreen  ,opacity=1 ]  {$\mc{Z}^{j}$};
\draw (537,245) node [anchor=north west][inner sep=0.75pt]  [color=\customred  ,opacity=1 ]  {$\mc{Z}^{i}$};
\draw (567,323) node [anchor=north west][inner sep=0.75pt]  [color=\customblue  ,opacity=1 ]  {$\mc{Z}^{k}$};

\end{scope}


\draw [color=\customgreen  ,draw opacity=1 ]   (72,481) .. controls (87.4,469.07) and (169.4,473.07) .. (236.2,479.87) .. controls (303.27,471.64) and (392.8,405.47) .. (514,406.17) ;
\draw [shift={(514,406.17)}, rotate = 0.33] [color=\customgreen  ,draw opacity=1 ][fill=\customgreen  ,fill opacity=1 ][line width=0.75]      (0, 0) circle [x radius= \radius, y radius= \radius]   ;
\draw [color=\customblue  ,draw opacity=1 ]   (72,481) .. controls (97.8,466.67) and (176.6,474.67) .. (306,485.87) .. controls (431,517.47) and (442.8,562.27) .. (560,569.17) ;
\draw [shift={(560,569.17)}, rotate = 3.37] [color=\customblue  ,draw opacity=1 ][fill=\customblue  ,fill opacity=1 ][line width=0.75]      (0, 0) circle [x radius= \radius, y radius= \radius]   ;
\draw [color=\customred  ,draw opacity=1 ]   (72,481) .. controls (112,451) and (491,520.17) .. (531,490.17) ;
\draw [shift={(531,490.17)}, rotate = 323.13] [color=\customred  ,draw opacity=1 ][fill=\customred  ,fill opacity=1 ][line width=0.75]      (0, 0) circle [x radius= \radius, y radius= \radius]   ;
\draw[fill=black] (68.63,481) circle [x radius= \radius, y radius= \radius]   ;
\draw (27,455) node [anchor=north west][inner sep=0.75pt]    {$z^{0}$};
\draw (519,380) node [anchor=north west][inner sep=0.75pt]  [color=\customgreen  ,opacity=1 ]  {$\mc{Z}^{j}$};
\draw (537,465) node [anchor=north west][inner sep=0.75pt]  [color=\customred  ,opacity=1 ]  {$\mc{Z}^{i}$};
\draw (566,543) node [anchor=north west][inner sep=0.75pt]  [color=\customblue  ,opacity=1 ]  {$\mc{Z}^{k}$};

\end{tikzpicture}

    \caption{Trajectories forming a tree-like structure (shown above) are optimal (with respect to problems \eqref{prb:gJmax}, \eqref{prb:l0min-i}, and \eqref{prb:l0min-best}) whereas the trajectories with irregular clumping (shown below) are not.}    
    \label{fig:tree-like}
\end{figure}
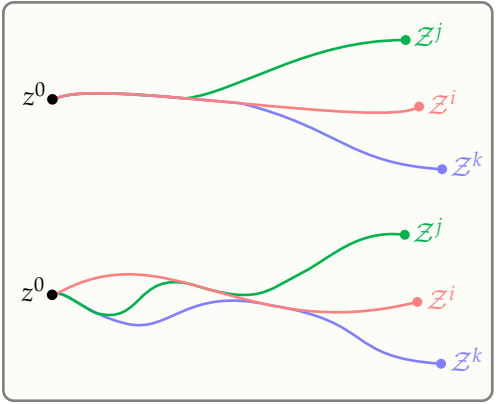
Given $i\in \bZr{1}{n}$, the problem of maximizing (in a cumulative sense) the number of reachable targets at each time instant from a trajectory to target $i$ can be described as follows
\begin{subequations}
\begin{align}
\underset{x^j_k,\,u^j_k}{\mr{minimize}}~&~\sum_{\genfrac{}{}{0pt}{1}{j\in [1:n]}{j\ne i}}\!\sum_{k=1}^{N} \znrm{x^i_k - x^j_k}, & & \label{eq:l0min-i-obj}\\
\mr{subject~to}~&~x^j_{k+1} = f(x^j_k,u^j_k), & &\!k\in\bZr{1}{N-1},\\
 &~x^j_k\in\bX,\,u^j_k\in\bU, & &\!k\in\bZr{1}{N-1},\\
 &~x^j_1 = z^0,~x^j_N\in \mc{Z}^j, & &\\
 &~j\in \bZr{1}{n}. & & \nonumber
\end{align}\label{prb:l0min-i}%
\end{subequations}
We can provide a mixed-integer representation for \eqref{prb:l0min-i} with binary variables, which we will explore in Section \ref{sec:soln-method}.

The following lemma relates the objective functions of \eqref{prb:opt-single-traj-i} and \eqref{prb:l0min-i}.
\begin{lemma}
Let trajectories $x^j_1,\ldots,x^j_N$, for $j\in \bZr{1}{n}$, form a solution to \eqref{prb:l0min-i}, and, for each $k\in\bZr{1}{N}$, let 
\begin{align}
    J_k \triangleq \big\{\,j\in \bZr{1}{n}\,\big|\,x^j_k = x^i_k\big\}.\label{eq:sets-l0min-i}    
\end{align}
Then the objective function \eqref{eq:l0min-i-obj} satisfies
\begin{align}
    \Bigg( \sum_{\genfrac{}{}{0pt}{1}{j\in [1:n]}{j\ne i}}\!\sum_{k=1}^{N} \znrm{x^i_k - x^j_k} \Bigg) + \sum_{k=1}^{N}|J_k| ={} & nN. \label{eq:l0min-i-obj-equiv-Jk}
\end{align}
\end{lemma}
\begin{proof} For each $k\in\bZr{1}{N}$, we have that 
\[
    \znrm{x^i_k-x^j_k} = \begin{cases}
        0 &\text{if }j\in J_k,\\
        1 &\text{otherwise.}
    \end{cases}
\]
Then 
\begin{align}
    \sum_{\genfrac{}{}{0pt}{1}{j\in [1:n]}{j\ne i}}\znrm{x^i_k - x^j_k} = \sum_{j\in [1:n]\setminus J_k}\znrm{x^i_k - x^j_k} = n - |J_k|.\label{eq:relate-cost-l0-J}
\end{align}
Taking the summation of \eqref{eq:relate-cost-l0-J} over $k\in\bZr{1}{N}$ provides \eqref{eq:l0min-i-obj-equiv-Jk}.
\end{proof}
\begin{lemma} The monotonicity property in Lemma \ref{lem:monotonicity} holds for the sets of target indices defined in \eqref{eq:sets-l0min-i}, i.e., $J_{k} \subseteq J_{k-1}$, for $k\in\bZr{2}{N}$.
\end{lemma}
\begin{proof} Let $x^j_1,\ldots,x^j_N$, for $j\in\bZr{1}{n}$, form a solution to \eqref{prb:l0min-i}, and sets $J_1,\ldots,J_N$ are given by \eqref{eq:sets-l0min-i}. Suppose that $J_{k}\setminus J_{k-1} \ne \emptyset$ for some $k\in\bZr{2}{N}$. Pick $j\in J_{k}\setminus J_{k-1}$. Then, $x^i_{k-1}\ne x^j_{k-1}$ and $x^i_k=x^j_k$. If $k=N$ replace the trajectory to target $j$ with the trajectory to target $i$, and if $k<N$ replace the trajectory to target $j$ with $x_1^i,\ldots,x_{k-1}^i,x_{k}^i,x^j_{k+1},\ldots,x^j_N$. Such a replacement reduces the value of objective function \eqref{eq:l0min-i-obj}. Therefore, we must have $J_{k}\setminus J_{k-1} = \emptyset$.
\end{proof}
\begin{theorem}\label{thm:l0min-i-equiv}If trajectories $x^j_1,\ldots,x^j_N$, for $j\in\bZr{1}{n}$, form a solution to \eqref{prb:l0min-i} and sets $J_1,\ldots,J_N$ are constructed via \eqref{eq:sets-l0min-i}, then $x^i_k,J_k$, for $k\in\bZr{1}{N}$, is a solution to \eqref{prb:opt-single-traj-i}. Conversely, if $x_k,J_k$, for $k\in\bZr{1}{N}$, is a solution to \eqref{prb:opt-single-traj-i}, then Corollary \ref{cor:construct-all-traj-from-single} provides a solution to \eqref{prb:l0min-i}. 
\end{theorem}
\begin{proof} 
Let trajectories $x^j_1,\ldots,x^j_N$, for $j\in\bZr{1}{n}$, solve \eqref{prb:l0min-i}, and let $\bar{x}_k,\bar{J}_k$, for $k\in\bZr{1}{N}$, solve \eqref{prb:opt-single-traj-i}. Then, we apply Corollary \ref{cor:construct-all-traj-from-single} to construct trajectories $\bar{x}^j_1,\ldots,\bar{x}^j_N$, for $j\in\bZr{1}{n}$, which are feasible with respect to \eqref{prb:l0min-i}. Suppose that $x^i_k,J_k$, for $k\in\bZr{1}{N}$, is feasible but not optimal with respect to \eqref{prb:opt-single-traj-i}, i.e.,
\[
    \sum_{k=1}^{N}|J_k| < \sum_{k=1}^{N}|\bar{J}_k|.
\]
From the equivalence of the objective functions \eqref{eq:opt-single-traj-obj} and \eqref{eq:l0min-i-obj}, shown via \eqref{eq:l0min-i-obj-equiv-Jk}, we conclude that trajectories $\bar{x}^j_1,\ldots,\bar{x}^j_N$, for $j\in\bZr{1}{n}$, lead to a strictly lower value of \eqref{eq:l0min-i-obj} than trajectories $x^j_1,\ldots,x^j_N$, for $j\in\bZr{1}{n}$. A contradiction.

Next, let $x_k,J_k$, for $k\in\bZr{1}{N}$, solve \eqref{prb:opt-single-traj-i}, and let trajectories $\bar{x}^j_1,\ldots,\bar{x}^j_N$, for $j\in\bZr{1}{n}$, solve \eqref{prb:l0min-i}. Then, $\bar{x}^i_k,\bar{J}_k$, for $k\in\bZr{1}{N}$, is feasible with respect to \eqref{prb:opt-single-traj-i}, where sets $\bar{J}_1,\ldots,\bar{J}_N$ are computed using \eqref{eq:sets-l0min-i}. Suppose that the trajectories constructed via Corollary \ref{cor:construct-all-traj-from-single} are feasible but not optimal with respect to \eqref{prb:l0min-i}. Therefore,
\[
    \sum_{\genfrac{}{}{0pt}{1}{j\in [1:n]}{j\ne i}}\!\sum_{k=1}^{N} \znrm{\bar{x}^i_k - \bar{x}^j_k} < \sum_{\genfrac{}{}{0pt}{1}{j\in [1:n]}{j\ne i}}\!\sum_{k=1}^{N} \znrm{x^i_k - x^j_k}
\]
From the equivalence of the objective functions \eqref{eq:opt-single-traj-obj} and \eqref{eq:l0min-i-obj} shown via \eqref{eq:l0min-i-obj-equiv-Jk}, we conclude that sets $\bar{J}_1,\ldots,\bar{J}_N$ lead to a strictly greater value of \eqref{eq:opt-single-traj-obj} than sets $J_1,\ldots,J_N$, which is a contradiction.
\end{proof}
\begin{remark}[Convex relaxation]\label{rem:cvx-relaxation}
The objective function \eqref{eq:l0min-i-obj} can be equivalently expressed using $0$-norm as follows
\begin{align*}
    & \sum_{\genfrac{}{}{0pt}{1}{j\in[1:n]}{j \ne i}} \sum_{k=1}^{N} \znrm{x^i_k - x^j_k} ={} \\
    & \sum_{k=1}^{N} \|\big( \|x^i_k - x^1_k\|_{p},\,\ldots,\,\|x^i_k - x^n_k\|_{p}\big)\|_0.
\end{align*}
where $\|\square\|_p$ denotes $p$-norm for any $p\ge 1$. Then a convex relaxation of \eqref{eq:l0min-i-obj} is given by
\begin{align*}
    & \sum_{k=1}^{N} \|\big( \|x^i_k - x^1_k\|_{p},\,\ldots,\,\|x^i_k - x^n_k\|_{p} \big)\|_1 ={} \\ 
    & \sum_{k=1}^{N} \sum_{\genfrac{}{}{0pt}{1}{j\in[1:n]}{j\ne i}}\|x^i_k - x^j_k\|_{p}. 
\end{align*}
In practice, choosing $p=1$ is desirable. That would encourage sparsity of the vector of concatenated pairwise differences of states: $(x_1^i-x_1^j,\ldots,x_N^i-x_N^j)$, for each $j\in\bZr{1}{n}\setminus\{i\}$. 
\end{remark}
\subsection{Maximize reachable targets from trajectory to an arbitrary target}
We can generalize \eqref{prb:l0min-i} by optimizing over $i\in \bZr{1}{n}$ to pick target $i^\star$. Among all targets, a trajectory to target $i^\star$ can have the maximum possible number of reachable targets at each time instant.
\begin{subequations}
\begin{align}
\underset{x^j_k,\,u^j_k}{\mr{minimize}}~&\min_{i\in[1:n]}\sum_{\genfrac{}{}{0pt}{1}{j\in[1:n]}{j\ne i}} \sum_{k=1}^{N} \znrm{x^i_k - x^j_k}, & & \label{eq:l0min-best-obj}\\
\mr{subject~to}~&~x^j_{k+1} = f(x^j_k,u^j_k), & &\!\!\!\!\!\!\!\!\!\!\!\!\!\!\!k\in\bZr{1}{N-1},\\
 &~x^j_k\in\bX,\,u^j_k\in\bU, & &\!\!\!\!\!\!\!\!\!\!\!\!\!\!\!k\in\bZr{1}{N-1},\\
 &~x^j_1 = z^0,~x^j_N \in \mc{Z}^j, & &\\
 &~j\in \bZr{1}{n}. & & \nonumber
\end{align}\label{prb:l0min-best}%
\end{subequations}
Given trajectories $x^j_1,\ldots,x^j_N$, for $j\in\bZr{1}{n}$, which solve \eqref{prb:l0min-best}, let 
\begin{align}
    J^\star_k \triangleq \big\{\,j\in\bZr{1}{n}\,\big|\,x^j_k = x^{i^\star}_k\big\},\label{eq:sets-l0min-best}   
\end{align}
for each $k\in\bZr{1}{N}$, where 
\begin{align}
i^\star \in \underset{i\in[1:n]}{\mr{argmin}}  \sum_{\genfrac{}{}{0pt}{1}{j\in[1:n]}{j\ne i}} \sum_{k=1}^{N} \znrm{x^i_k - x^j_k}.\label{eq:istar-l0min-best} 
\end{align}
\begin{remark}\label{rem:take-min-i-outside} Solving \eqref{prb:l0min-best} is equivalent to solving \eqref{prb:l0min-i} for each target $i\in\bZr{1}{n}$ and picking the one which leads to least value for \eqref{eq:l0min-i-obj}. Similarly, solving \eqref{prb:opt-single-traj} is equivalent to solving \eqref{prb:opt-single-traj-i} for each $i\in\bZr{1}{n}$ in the boundary condition \eqref{eq:opt-pick-targ-i}, and selecting the best solution.
\end{remark}
\begin{theorem} Let trajectories $x^j_1,\ldots,x^j_N$, for $j\in\bZr{1}{n}$, solve \eqref{prb:l0min-best}. Sets $J^\star_1,\ldots,J^\star_N$ and $i^\star$ are given by \eqref{eq:sets-l0min-best} and \eqref{eq:istar-l0min-best}, respectively. Then $x^{i^\star}_k,J^\star_k$, for $k\in\bZr{1}{N}$, is a solution to \eqref{prb:opt-single-traj}. Conversely, if $x_k,J_k$, for $k\in\bZr{1}{N}$, is a solution to \eqref{prb:opt-single-traj}, then Corollary \ref{cor:construct-all-traj-from-single} provides a solution to \eqref{prb:l0min-best}.
\end{theorem}
\begin{proof}
With the observation in Remark \ref{rem:take-min-i-outside}, the proof is similar to that of Theorem \ref{thm:l0min-i-equiv}.
\end{proof}
\begin{remark}
The branch times and branch points (described in Definition \ref{def:branch-kj}) can be computed for the solutions to \eqref{prb:l0min-i} and \eqref{prb:l0min-best} using the targets sets in \eqref{eq:sets-l0min-i} and \eqref{eq:sets-l0min-best}, respectively. The branch time $k^j$ is the time after which the trajectory to target $j$ ``detaches'' from the trajectory to target $i$ (or $i^\star$).
\end{remark}
We have established a tree-like structure for the trajectories that solve the optimization problems in Section \ref{sec:card-min} and highlighted their connection to the monotonicity of sets $J_1,\ldots,J_N$, which solve the optimization problems in Section \ref{sec:reach}. In other words, it is not optimal for trajectories to clump together arbitrarily as shown in Figure \ref{fig:tree-like}.
\section{Solution Method}\label{sec:soln-method}
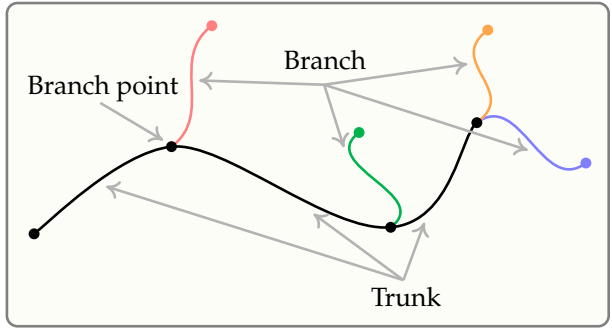
\begin{figure}[!ht]
\centering



\begin{tikzpicture}[x=0.75pt,y=0.75pt,yscale=-1,xscale=1,scale=0.55]

\def\radius{5}

\node[rectangle, fill=beige!20, draw=black!50, minimum width=8cm, minimum height=4.3cm, rounded corners, line width=1pt] at (365,237) {};

\tikzset{
    customArrow/.style={-{To[length=5pt, width=7pt]}}
}

\draw [color=\customblue  ,draw opacity=1 ]   (520,198.5) .. controls (560,168.5) and (580,265.5) .. (620,235.5) ;
\draw  [draw opacity=0][fill=\customblue  ,fill opacity=1 ] (620,235.5) circle [x radius=\radius, y radius=\radius];

\draw [color=\customorange  ,draw opacity=1 ]   (520,198.5) .. controls (560,168.5) and (490,143.5) .. (530,113.5) ;
\draw  [draw opacity=0][fill=\customorange  ,fill opacity=1 ] (530,113.5) circle [x radius=\radius, y radius=\radius] ;

\draw [color=\customred  ,draw opacity=1 ]   (240,220.5) .. controls (280,190.5) and (237,139.5) .. (277,109.5) ;
\draw [draw opacity=0][fill=\customred  ,fill opacity=1 ] (277,109.5) circle [x radius=\radius, y radius = \radius] ;

\draw [color=\customgreen  ,draw opacity=1 ]   (441,294.5) .. controls (481,264.5) and (372,237.5) .. (412,207.5) ;
\draw  [draw opacity=0][fill=\customgreen  ,fill opacity=1 ] (412,207.5) circle [x radius=\radius, y radius=\radius] ;

\draw    (114,300.5) .. controls (126.01,291.5) and (191.3,225.53) .. (240,220.5) .. controls (288.7,215.47) and (383.9,299.42) .. (441,294.5) .. controls (498.1,289.58) and (504.94,209.8) .. (520,198.5) ;
\draw [draw opacity=0][fill=black ,fill opacity=1 ] (520,198.5) circle [x radius=\radius, y radius=\radius] ;
\draw [draw opacity=0][fill=black ,fill opacity=1 ] (114,300.5) circle [x radius=\radius, y radius=\radius] ;
\draw [draw opacity=0][fill=black ,fill opacity=1 ] (441,294.5) circle [x radius=\radius, y radius=\radius] ;
\draw [draw opacity=0][fill=black ,fill opacity=1 ] (240,220.5) circle [x radius=\radius, y radius=\radius] ;

\draw [color=black!30  ,draw opacity=1, customArrow ]   (380,163.67) -- (398.07,221.09) ;

\draw [color=black!30  ,draw opacity=1, customArrow ]   (380,163.67) -- (511.35,144.62) ;

\draw [-, color=black!30  ,draw opacity=1, customArrow ]   (380,163.67) -- (566.76,223.72) ;

\draw [color=black!30  ,draw opacity=1, customArrow ]   (380,163.67) -- (265.33,159.74) ;

\draw [color=black!30  ,draw opacity=1, customArrow ]   (452.67,343) -- (179.91,256.94) ;

\draw [color=black!30  ,draw opacity=1, customArrow ]   (452.67,343) -- (370.28,282.19) ;

\draw [color=black!30  ,draw opacity=1, customArrow ]   (452.67,343) -- (471.98,290.21) ;

\draw [color=black!30  ,draw opacity=1, customArrow ]   (174.67,180.33) -- (232.28,214.64) ;

\draw (340,129) node [anchor=north west][inner sep=0.75pt]   [align=left] {Branch};
\draw (420,346) node [anchor=north west][inner sep=0.75pt]   [align=left] {Trunk};
\draw (105,152) node [anchor=north west][inner sep=0.75pt]   [align=left] {Branch point};

\end{tikzpicture}

\caption{Algorithms \ref{alg:ddtoqcvx} and \ref{alg:ddtoscp} recursively compute trunk and branch trajectories connected by branch points while adhering to the given target prioritization.}
\label{fig:branch-trunk}
\end{figure}
This section describes three solution methods for {\ddto} that either numerically solve optimization problems in Section \ref{sec:card-min} under special cases, or generate solutions with the tree structure in the general setting. Two of the solution methods are specialized for discrete-time affine systems subject to convex constraints on the state and control input. The first method, {\ddtoqcvx}, is based on the computation of the latest branch time (via \eqref{prb:gJmax}) by solving quasiconcave maximization problems. Since \eqref{prb:gJmax} is agnostic to prioritization of targets, we adopt a recursive approach described in Algorithm \ref{alg:ddtoqcvx}, wherein the latest branch times for shrinking collections of targets are sequentially computed. At each branch time, at least one target is rejected based on a given priority order: $\lambda^k$, for $k\in\bZr{1}{n}$, where $\lambda^1$ denotes the highest priority target and $\lambda^n$ denotes the least priority. Two consecutive branch points are connected by a trajectory segment called \textit{trunk}, where the trajectories to a collection of targets are coincident. At each branch time, the number of trajectories in the trunk reduces by at least one and the trajectory segment which ``detaches'' from the trunk to terminate at the ``rejected'' target is called a \textit{branch}. The sequence of branches and trunks possess a tree-like structure. See Figure \ref{fig:branch-trunk} for an illustration of branches and trunks. Note that it is possible for two successive branch times to coincide. In which case, more than one target is rejected at the same branch time. We refer the reader to \cite{elango2022deferring} for an extension to Algorithm \ref{alg:ddtoqcvx} with adaptive update of the constraints on cumulative trajectory cost.\\ The second method, {\ddtomicp}, consists of an MICP reformulation of \eqref{prb:l0min-i} (along with its convex relaxation), which are solved with off-the-shelf MICP solvers and (continuous) convex optimization solvers. \\ The final solution method, {\ddtoscp}, is designed for a more general setting: continuous-time nonlinear systems with nonconvex path constraints. It borrows the recursive approach and mimics the tree structure of the trajectories in {\ddtoqcvx}, as shown in Algorithm \ref{alg:ddtoscp}.
\subsection{{\ddtomicp}}\label{subsec:ddtomicp}
For an affine dynamical system with convex state and control input constraints, \eqref{prb:l0min-i} can be written as an MICP with binary variables. In the development thus far, we assumed that the trajectories to all targets have the same horizon length, with the understanding that embedded time-dilation (described in Appendix \ref{app:dilation}) can allow for different final times for each of the trajectories. However, time-dilation introduces nonlinearity in the system dynamics. So, to avoid time-dilation and preserve convexity, we allow the trajectories to different targets to have different horizon lengths, i.e., $N^j$ for $j\in\bZr{1}{n}$. Similarly, since the cumulative trajectory constraint function is typically nonlinear, we avoid augmenting such a constraint into the system dynamics (with the approach in Remark \ref{rem:aug-sys-cost}) to preserve convexity.
\begin{assumption}\label{asm:traj-diff-bnd}
The norm of the difference between states at each time on any two trajectories is bounded above by $M>0$.
\end{assumption}
For each $j\in \bZr{1}{n}$, let $x^j_1,\ldots,x^j_N$ be a trajectory to $\mc{Z}^j$. The key step in formulating an MICP is as follows. Let $i\in\bZr{1}{n}$, $p\ge 1$ and $N^{ij} \triangleq \min\{N^i,N^j\}$, for $j\in\bZr{1}{n}$. For each $j\in\bZr{1}{n}\setminus \{i\}$ and $k\in\bZr{1}{N^{ij}}$ we use a binary variable $\zeta^j_k\in\{0,1\}$ to model the implication 
\[
    \|x^i_k-x^j_k\|_p > 0 \implies \zeta^j_k = 1,
\]
as follows
\begin{align}
    \|x^i_k-x^j_k\|_p \le \zeta^j_k M.\label{eq:micp-cnstr}
\end{align}
Due to Assumption \ref{asm:traj-diff-bnd}, imposing \eqref{eq:micp-cnstr} in an optimization problem will not cause infeasibility. \\ The MICP representation for \eqref{prb:l0min-i} in the convex case is given by
\begin{subequations}
\begin{align}
    \underset{x^j_k,\,u^j_k,\,\zeta^j_k}{\mr{minimize}}~&\sum_{\genfrac{}{}{0pt}{1}{j\in [1:n]}{j\ne i}} \sum_{k=1}^{N^{ij}} \zeta^j_k, & &\\
    \mr{subject~to}~&x^j_{k+1} = Ax^j_k + Bu^j_k + c, & &\hspace{-0.9cm}k\in\bZr{1}{N^j-1},\\
    &x^j_k\in\bX,\,u^j_k\in\bU, & &\hspace{-0.9cm}k\in\bZr{1}{N^j-1},\\
    &\sum_{k=1}^{N^j-1} l(x^j_k,u^j_k) \le l_{\max}, & & \\
    &x^j_1 = z^0,~x^j_N\in \mc{Z}^j, & &\\
    &j\in\bZr{1}{n}, & & \nonumber\\
    &\|x^i_k-x^j_k\|_p \le \zeta^j_k M, & &\\
    &\zeta^j_k\in\{0,1\}, & &\\
    &k\in\bZr{1}{N^{ij}},~j\in\bZr{1}{n}\setminus\{i\}, & & \nonumber
\end{align}\label{prb:micp-i}%
\end{subequations}
where $A\in\bR^{n_x\times n_x}$, $B\in\bR^{n_x\times n_u}$, and $c\in\bR^{n_x}$ define the affine dynamical system. We can use efficient state-of-the-art MICP solvers such as MOSEK and Gurobi to solve \eqref{prb:micp-i}. The tree structure established in Section \ref{sec:card-min} for a solution to \eqref{prb:l0min-i} ensures that $\zeta^j_k$ monotonically increases from $0$ to $1$. As an alternative to directly solving an MICP, we can efficiently obtain an approximate solution for \eqref{prb:micp-i} using a conic optimization solver by replacing $\zeta^j_k\in\{0,1\}$ in \eqref{prb:micp-i} with its convex relaxation: $\zeta^j_k\in[0,1]$. 
\subsection{{\ddtoqcvx}}
\begin{lemma} For an affine dynamical system with convex state and control input constraints, \eqref{prb:gJmax} is a quasiconvex optimization problem, i.e., when $f$ is an affine function, constraint sets $\mathbb{X}$, $\mathbb{U}$, and targets $\mc{Z}^j$, for $j\in\bZr{1}{n}$, are closed and convex. 
\end{lemma}
\begin{proof} When $f$ is an affine function, $\mathbb{X}$ and $\mathbb{U}$ are closed convex sets, the feasible set of \eqref{prb:gJmax} is convex. The objective function \eqref{eq:gJmax-qcvx-obj} is quasiconcave because its superlevel sets are convex sets (hyperplanes), i.e.,
\begin{align*}
     & \gfunc^{|J|}(X^J) \ge k^\star\\
\iff & x^j_k = x^i_k,~\forall\,k\in\bZr{1}{k^\star},~i,j\in J\\
\iff & \sum_{m=1}^{k-1}A^{k-1-m}B(u^j_m-u^i_m)=0,~\forall\,k\in\bZr{1}{k^\star},~i,j\in J
\end{align*}
\end{proof}
We can efficiently solve \eqref{prb:gJmax} via bisection method, which solves a sequence of convex feasibility problems \cite[Sec. 3]{agrawal2020disciplined}, \cite[Alg. 1]{elango2022deferring}. The sequence is guaranteed to terminate within a fixed number of iterations. 

Similar to {\ddtomicp}, we preserve convexity by 1) allowing different horizon lengths for each of the trajectories (to avoid time-dilation), and 2) explicitly imposing the cumulative trajectory constraint. Then, under the {\ddtoqcvx} specialization, the solution method described in Algorithm \ref{alg:ddtoqcvx} recursively solves the following quasiconvex problem.
\begin{subequations}
\begin{align}
    \underset{x^j_k,\,u^j_k}{\mr{maximize}}~&\gfunc^{|J|}(X^J), & & \\
    \mr{subject~to}~&x^j_{k+1}\!=\!Ax^j_k+Bu^j_k+c, & &\hspace{-0.1cm}k\in\bZr{1}{M^j\!-\!1},\label{eq:ddto-qcvx-dyn}\\
    &x^j_k\in\bX,~u^j_k\in\bU, & &\hspace{-0.1cm}k\in\bZr{1}{M^j\!-\!1},\\
    &\sum_{k=1}^{M^j-1}l(x^j_k,u^j_k) \le l_{\max}, & &\label{eq:ddto-qcvx-cumcostcnstr}\\
    &x_1^j = z^0,~x^j_{M^j}\in \mc{Z}^j, & &\\
    &j\in J, & & \nonumber
\end{align}\label{prb:ddto-qcvx}%
\end{subequations}
where the trajectories to the targets in $J$ are concatenated into vector $X^J$. The set of target indices $J\subseteq\bZr{1}{n}$, horizon lengths $M^j$, for $j\in J$, and the initial state $z^0$, are updated after each branch time computation. The optimal value of \eqref{prb:ddto-qcvx} is a branch time and the segment of any of the trajectories until the branch time forms a trunk.
\subsection{{\ddtoscp}}
Consider a continuous-time nonlinear dynamical system
\begin{align}
    \dot{x}(t) ={} F(x(t),u(t)),\quad t\in[0,\tfinal].
\end{align}
The state and input constraints: $x(t)\in\bX$ and $u(t)\in\bU$ can be re-expressed as path constraints: $g_i(x(t),u(t)) \le 0$, for $i\in\bZr{1}{n_g}$, and $h_j(x(t),u(t)) = 0$, for $j\in\bZr{1}{n_h}$, where $g_i$ and $h_j$ are scalar-valued functions. We consider a free-final-time optimal control problem, i.e., the final time $\tfinal$ is a decision variable. So, we adopt time-dilation \cite[Sec. 2.4]{elango2024successive} which treats the actual time $t$ as a continuously differentiable, strictly-increasing mapping from a known normalized interval $[0,1]$ to the actual time interval $[0,\tfinal]$. The dilation factor, given by
\begin{align}
    s(\tau) \triangleq{} & \derv{t}(\tau) = \frac{\mr{d}t(\tau)}{\mr{d}\tau},\quad \tau\in[0,1],
\end{align}
is treated as an additional control input. Next, the path constraints are subjected to the isoperimetric reformulation \cite[Sec. 2.3]{elango2024successive}. For any $i\in\bZr{1}{n_g}$ and $j\in\bZr{1}{n_h}$, we have that
\begin{align*}
    &\!g_i(x(t),u(t)) \le 0,~h_j(x(t),u(t))=0,\quad\forall\,t\in[0,\tfinal],\\
    \iff &\!\int_0^{\tfinal}\normplus{g_i(x(t),u(t))}^2+h_j(x(t),u(t))^2\mr{d}t = 0.
\end{align*}
For each $\tau\in[0,1]$, we augment the state and control input as follows
\begin{align*}
    \tilde{x}(\tau) \triangleq{} & (x(t(\tau)),y(\tau),t(\tau)),\\
    \tilde{u}(\tau) \triangleq{} & (u(t(\tau)),s(\tau)),
\end{align*}
to obtain the following augmented dynamical system defined over interval $[0,1]$
\begin{align}
    \derv{\tilde{x}} ={} & s\begin{bmatrix}
        F(x,u)\\
        \displaystyle\sum_{i=1}^{n_g}\normplus{g_i(x,u)}^2+\sum_{j=1}^{n_h}h_j(x,u)^2\\
        1        
    \end{bmatrix} = \tilde{F}(\tilde{x},\tilde{u}).\label{eq:aug-sys-ct}
\end{align}
Imposing periodicity boundary conditions on $y$ is equivalent to satisfying the path constraints in continuous-time. The above constraint reformulation approach is especially useful within direct methods for trajectory optimization to avoid inter-sample constraint violation which is commonly encountered after discretization and imposing path constraint at finitely-many node points. We refer the reader to \cite{elango2024successive} for detailed discussions.

Given a set of target indices $J\subseteq\bZr{1}{n}$, initial state $z^0$, and trajectory horizon length $M$, {\ddtoscp} considers the following continuous-time optimal control problem for $|J|+1$ trajectories: trunk trajectory $x^0$, defined over interval $[0,\tfinal^0]$, and branch trajectories $x^j$, defined over interval $[0,\tfinal^j]$, for each $j\in J$. 
\begin{subequations}
\begin{align}
    \underset{\tilde{x}^j,\,\tilde{u}^j}{\mr{maximize}}~&\selector{t}\tilde{x}^0(1)\\
    \mr{subject~to}~&\derv{\tilde{x}}{}^j(\tau)\!=\tilde{F}(\tilde{x}^j(\tau),\tilde{u}^j(\tau)),& &\hspace{-1.22cm}\tau\in[0,1],\label{eq:ddto-scp-dyn}\\
    & \tilde{u}^j(\tau) \in \tilde{\mc{U}}, & &\hspace{-1.22cm}\tau\in[0,1],\\
    & \selector{y}\tilde{x}^j(0) = \selector{y}\tilde{x}^j(1), & &\label{eq:ddto-scp-cnstrbc}\\
    & j\in J\cup\{0\}, & & \nonumber\\
    & \selector{x}\tilde{x}^j(0) = \selector{x}\tilde{x}^0(1), & & \\
    & P^j(\selector{x}\tilde{x}^j(1)) \le 0,\,Q^j(\selector{x}\tilde{x}^j(1)) = 0,& & \label{eq:ddto-scp-bc}\\
    & j\in J, & & \nonumber\\
    & \selector{x}\tilde{x}^0(0) = z^0. & &    
\end{align}\label{prb:ddto-scp}%
\end{subequations}
Since the targets can be nonconvex sets, we express the terminal boundary condition constraints using continuously differentiable constraint functions $P^j:\bR^{n_x}\to\bR^{n_P}$ and $Q^j:\bR^{n_x}\to\bR^{n_Q}$. Note that a cumulative constraint on the trajectory is transformed to a terminal constraint in \eqref{eq:ddto-scp-bc} through an approach similar to that in Remark \ref{rem:aug-sys-cost}. By construction, $\tfinal^0$ is the branch time and the goal is to maximize it. Convex constraints on the control input are encoded through the convex and compact set $\tilde{\mc{U}}$. We use selector matrices $\selector{x}$, $\selector{y}$, and $\selector{t}$ to select components of $\tilde{x}$ corresponding to $x$, $y$, and $t$, respectively. We assume that all functions appearing in \eqref{prb:ddto-scp} are continuously differentiable. Note that, in contrast to {\ddtoqcvx}, there is an explicit distinction between the trunk and branch trajectories in \eqref{prb:ddto-scp}.

Given a solution to \eqref{prb:ddto-scp}, the trajectory from $z^0$ to target $j\in J$, defined over interval $[0,\tfinal^0+\tfinal^j]$, is obtained as follows
\begin{align*}
    x^j(t) \triangleq{} & \!\!\begin{cases}
                    \selector{x}\tilde{x}^0((\selector{t}\tilde{x}^0)^{-1}(t))&\text{if }t\in[0,\tfinal^0],\\
                    \selector{x}\tilde{x}^j( (\selector{t}\tilde{x}^j)^{-1}(t-\tfinal^0))&\text{if }t\in(\tfinal^0,\tfinal^0+\tfinal^j],
                 \end{cases}
\end{align*}
where $\tfinal^j = \selector{t}\tilde{x}^j(1)$. Note that $\selector{t}\tilde{x}^j : [0,1]\to [0,\tfinal^j]$ is invertible since it is a strictly increasing function.

Next, to numerically solve \eqref{prb:ddto-scp}, we discretize $[0,1]$ with a uniformly-spaced grid of length $M$: $0=\tau_1<\ldots<\tau_M=1$. For each $j\in J\cup\{0\}$, the values of the augmented state and control input at the nodes of the grid: $\tilde{x}^j_1,\ldots,\tilde{x}^j_M$ and $\tilde{u}^j_1,\ldots,\tilde{u}^j_{M}$, respectively, are treated as decision variables. We use a zero-order-hold parametrization for the augmented control input, $\tilde{\nu}^j : [0,1]\to\bR^{n_u+1}$, given by $\tilde{\nu}^j(\tau) = \tilde{u}^j_k$, whenever $\tau\in[\tau_k,\tau_{k+1})$ for some $k\in\bZr{1}{M-1}$, and $\tilde{\nu}^j(1) = \tilde{u}^j_M$. \\ Then, for each $k\in\bZr{1}{M-1}$, the exact discretization of \eqref{eq:ddto-scp-dyn} via multiple shooting \cite{bock1984multiple} is  
\begin{align}
    \tilde{x}^j_{k+1} ={} & \tilde{f}(\tilde{x}^j_k,\tilde{u}^j_k) \triangleq \tilde{x}^j_k+\int_{\tau_{k}}^{\tau_{k+1}}\tilde{F}({}^k\tilde{x}^j(\tau),\tilde{\nu}^j(\tau))\mr{d}\tau,
\end{align}
where ${}^k\tilde{x}^j$ is a trajectory for \eqref{eq:ddto-scp-dyn} over $[\tau_k,\tau_{k+1}]$ with control input $\tilde{\nu}^j$ and initial condition $\tilde{x}^j_k$.

We obtain the following nonconvex optimization problem after time-dilation, constraint reformulation, augmented control input parametrization, and multiple-shooting discretization.
\begin{subequations}
\begin{align}
    \underset{\tilde{x}^j_k,\,\tilde{u}^j_k}{\mr{maximize}}~&~\selector{t}\tilde{x}^0_M & & \\
    \mr{subject~to}~&~\tilde{x}_{k+1}^j = \tilde{f}(\tilde{x}^j_k,\tilde{u}^j_k), & &\hspace{-1.36cm}k\in\bZr{1}{M\!-\!1},\\
    &~\tilde{u}^j_k\in\tilde{\mc{U}}, & &\hspace{-1.36cm}k\in\bZr{1}{M\!-\!1},\\
    &~\selector{y}(\tilde{x}^j_{k+1}-\tilde{x}^j_{k})\le \epsilon, & &\hspace{-1.36cm}k\in\bZr{1}{M\!-\!1},\label{eq:ddto-scp-disc-cnstrrelax}\\
    &~j\in J\cup\{0\}, & & \nonumber\\
    &~\selector{x}\tilde{x}^j_1 = \selector{x}\tilde{x}^0_M, & &\\
    &~P^j(\selector{x}\tilde{x}^j_M) \le 0,\,Q^j(\selector{x}\tilde{x}^j_M) = 0, & &\\
    &~j\in J, & & \nonumber\\
    &~\selector{x}\tilde{x}^0_1 = z^0. & &
\end{align}\label{prb:ddto-scp-disc}%
\end{subequations}
The augmented trajectory to target $j\in J$ is represented with $\tilde{x}^0_1,\ldots,\tilde{x}^0_M,\tilde{x}^j_2,\ldots,\tilde{x}^j_{M}$. Let  $N \triangleq 2M-1$ denote the total horizon length. The corresponding discrete-time trajectory and control input sequence are given by 
\begin{subequations}
\begin{align}
    x^j_{k} \triangleq {} & \begin{cases}
        \selector{x}\tilde{x}^0_{k}&\text{if }k\in\bZr{1}{M},\\
        \selector{x}\tilde{x}^j_{k-M+1}&\text{if }k\in\bZr{M+1}{N},
    \end{cases}\\
    u^j_{k} \triangleq {} & \begin{cases}
        \selector{u}\tilde{u}^0_{k}&\text{if }k\in\bZr{1}{M-1},\\
        \selector{u}\tilde{u}^j_{k-M+1}&\text{if }k\in\bZr{M}{N-1}.
    \end{cases}
\end{align}
\end{subequations}
Note that the equality constraint \eqref{eq:ddto-scp-cnstrbc} is relaxed to \eqref{eq:ddto-scp-disc-cnstrrelax} to avoid automatic violation of constraint qualifications (see \cite[Sec. 3.1]{elango2024successive}). We solve \eqref{prb:ddto-scp-disc} using {\ctscvx}, an SCP-based nonconvex trajectory optimization framework \cite{elango2024successive}.
\subsection{Algorithm}
The {\ddto} solution methods described in Algorithms \ref{alg:ddtoqcvx} and \ref{alg:ddtoscp}, recursively solve \eqref{prb:ddto-qcvx} and \eqref{prb:ddto-scp-disc}, respectively. Both methods are myopic in their computation of branch times, i.e., they are computed sequentially instead of simultaneously, such as in {\ddtomicp} by solving \eqref{prb:micp-i}. Furthermore, after each branch time computation, the cumulative trajectory constraint must be updated to account for the contribution due to the previous trunk segment (i.e., $l_{\max}$ in \eqref{eq:ddto-qcvx-cumcostcnstr} must be updated). On the other hand, besides handling target prioritization, the recursive nature of the solution methods is beneficial within a closed-loop simulation (as demonstrated in \cite{haynerbuckner2023halo}) wherein the computation of the latest branch time can use new information acquired by perception. 

Algorithms \ref{alg:ddtoqcvx} and \ref{alg:ddtoscp} take initial state, targets, target prioritization, and trajectory horizon length(s), to provide branch and trunk trajectories. The total trajectory length $N$ passed to Algorithm \ref{alg:ddtoscp} is assumed to be an odd number.
\begin{algorithm}[!ht]
    \caption{{\ddtoqcvx}}
    \label{alg:ddtoqcvx}
    \begin{algorithmic}[1]
        \Require $z^0$, $\mc{Z}^j$, $\lambda^j$, $N^j$, for $j\in\bZr{1}{n}$
        \State $k^{\lambda^{n+1}}\gets 1$, $J_1\gets\bZr{1}{n}$
        \For{$k\in\bZr{1}{n-1}$}
            \State $M^j\gets N^j-k^{\lambda^{n-k+1}}$, for $j\in J_k$
            \State Solve \eqref{prb:ddto-qcvx} via bisection method for target set $J_k$
            \Statex\hspace{0.5cm}with initial state $z^0$ and horizon lengths $M^j$
            \State Store $k^{\lambda^{n-k+1}}$\Comment{Branch time}
            \LineComment{0.5cm}{Trunk trajectory and control input sequence}
            \State Store $x^0_k$, for $k\in\bZr{k^{\lambda^{n-k+2}}}{k^{\lambda^{n-k+1}}}$
            \State Store $u^0_k$, for $k\in\bZr{k^{\lambda^{n-k+2}}}{k^{\lambda^{n-k+1}}\!-\!1}$
            \LineComment{0.5cm}{Branch trajectory and control input sequence}
            \State Store $x^{\lambda^{n-k+1}}_k$, for $k\in\bZr{k^{\lambda^{n-k+1}}}{N^{\lambda^{n-k+1}}}$
            \State Store $u^{\lambda^{n-k+1}}_k$, for $k\in\bZr{k^{\lambda^{n-k+1}}}{N^{\lambda^{n-k+1}}\!-\!1}$
            \If{$k=n-1$}
                \State $k^{\lambda^1}\gets k^{\lambda^2}$
                \State Store $x^{\lambda^{1}}_k$, for $k\in\bZr{k^{\lambda^{2}}}{N^{\lambda^{1}}}$
                \State Store $u^{\lambda^{1}}_k$, for $k\in\bZr{k^{\lambda^{2}}}{N^{\lambda^{1}}\!-\!1}$
            \EndIf
            \State $J_{k+1}\gets J_k\setminus\{\lambda^{n-k+1}\}$\Comment{Reject target}
            \State $z^0\gets x^0_{k^{n-k+1}}$\Comment{Branch point}
            \State Update cumulative trajectory constraint
        \EndFor 
        \Ensure $x_k^0$, $u_k^0$, for $k\in\bZr{1}{k^{\lambda^2}}$,
        \Statex\hspace{0.85cm}$x_k^{\lambda^j}$, $u^{\lambda^j}_k$, for $k\in\bZr{k^{\lambda^j}}{N^{\lambda^j}}$, $j\in\bZr{1}{n}$
    \end{algorithmic}
\end{algorithm}
\begin{algorithm}[!ht]
    \caption{{\ddtoscp}}
    \label{alg:ddtoscp}
    \begin{algorithmic}[1]
        \Require $z^0$, $\mc{Z}^j$, $\lambda^j$, for $j\in\bZr{1}{n}$, $N$
        \State $t^{\lambda^{n+1}}\gets 0$, $J_1\gets\bZr{1}{n}$
        \State $M \gets N+1$
        \For{$k\in\bZr{1}{n-1}$}
            \State $M \gets \lceil M/2 \rceil$
            \State Solve \eqref{prb:ddto-scp-disc} via {\ctscvx} for target set $J_k$
            \Statex\hspace{0.5cm}with initial state $z^0$ and horizon length $M$
            \State $t^{\lambda^{n-k+1}}\gets t^{\lambda^{n-k+2}} + \selector{t}\tilde{x}^0(1)$\Comment{Branch time}
            \State $\tfinal^{\lambda^{n-k+1}}\gets t^{\lambda^{n-k+1}} + \selector{t}\tilde{x}^{\lambda^{n-k+1}}(1)$
            \LineComment{0.5cm}{Trunk trajectory and control input}
            \State Store $x^0(t)$, for $t\in[t^{\lambda^{n-k+2}},t^{\lambda^{n-k+1}}]$
            \State Store $u^0(t)$, for $t\in[t^{\lambda^{n-k+2}},t^{\lambda^{n-k+1}}]$
            \LineComment{0.5cm}{Branch trajectory and control input sequence}
            \State Store $x^{\lambda^{n-k+1}}(t)$, for $t\in[t^{\lambda^{n-k+1}},\tfinal^{\lambda^{n-k+1}}]$
            \State Store $u^{\lambda^{n-k+1}}(t)$, for $t\in[t^{\lambda^{n-k+1}},\tfinal^{\lambda^{n-k+1}}]$
            \If{$k=n-1$}
                \State $t^{\lambda^1}\gets t^{\lambda^2}$
            \State Store $x^{\lambda^{1}}(t)$, for $t\in[t^{\lambda^{2}},\tfinal^{\lambda^{1}}]$
            \State Store $u^{\lambda^{1}}(t)$, for $t\in[t^{\lambda^{2}},\tfinal^{\lambda^{1}}]$
            \EndIf
            \State $J_{k+1}\gets J_k\setminus\{\lambda^{n-k+1}\}$\Comment{Reject target}
            \State $z^0\gets x^0(t^{\lambda^{n-k+1}})$\Comment{Branch point}
            \State Update cumulative trajectory constraint
        \EndFor 
        \Ensure $x^0(t),u^0(t)$, for $t\in[0,t^{\lambda^2}]$,
        \Statex\hspace{0.85cm}$x^{\lambda^j}(t),u^{\lambda^j}(t)$, for $t\in[t^{\lambda^j},\tfinal^{\lambda^j}]$, $j\in\bZr{1}{n}$
    \end{algorithmic}
\end{algorithm}
\section{Numerical Results}\label{sec:num-results}
This section demonstrates {\ddtoqcvx} (Algorithm \ref{alg:ddtoqcvx}), {\ddtomicp}, and {\ddtoscp} (Algorithm \ref{alg:ddtoscp}), using two optimal control applications based on quadrotor motion planning (described in Appendix \ref{app:ocp-eg}). To ensure reliable numerical performance of the solution methods, we scale the primal variables and path constraints functions so that the values that they take are of similar orders of magnitude. The code used to generate the numerical results is provided at:
\begin{center}
\url{https://github.com/purnanandelango/ddto}
\end{center}
%
%
Figures \ref{fig:ddto-qcvx} and \ref{fig:ddto-micp} show the results of Algorithm \ref{alg:ddtoqcvx} and {\ddtomicp}, respectively, for a discrete-time convex optimal control example described in Appendix \ref{app:ocp-eg-dt-cvx}. Figure \ref{fig:ddto-scp} shows the result of Algorithm \ref{alg:ddtoscp} for a continuous-time nonconvex optimal control example described in Appendix \ref{app:ocp-eg-ct-ncvx}. Both examples, which relate to quadrotor motion planning, consider four target states with a specified priority order. The first trunk is denoted with \inlinecircle[2.2pt]{1pt}{trunkone}, the second trunk with \inlinesquare[4pt]{1pt}{trunktwo}, and the third trunk with \inlinetriangle[5pt]{1pt}{trunkthree}. The branches are denoted with \inlinedot[2pt]{1pt}{\customorange}, \inlinedot[2pt]{1pt}{\customblue}, \inlinedot[2pt]{1pt}{\customgreen}, and \inlinedot[2pt]{1pt}{\customred}. The constraint bounds are denoted with \tikz[baseline=(X.base)] \draw[cyan!50,line width=2pt] (0,1pt) -- (5pt,1pt) node (X) at (2pt,-2pt) {};. The nodes in Figure \ref{fig:ddto-qcvx} correspond to the discrete-time state and control input, whereas the nodes in Figure \ref{fig:ddto-scp} denote the SCP solution variables at the discretization nodes.

The most preferred target, $z^1$, plays the role of target $i$ in \eqref{prb:micp-i}. Observe that the branch times in the {\ddtomicp} solution are more delayed than in the {\ddtoqcvx} solution. While {\ddtomicp} can provide greater overall deferrability than {\ddtoqcvx}, it cannot however enforce a target preference order.
\begin{figure}[!ht]
    \centering
    \begin{subfigure}[b]{\linewidth}
    \centering
    \includegraphics[width=0.9\linewidth]{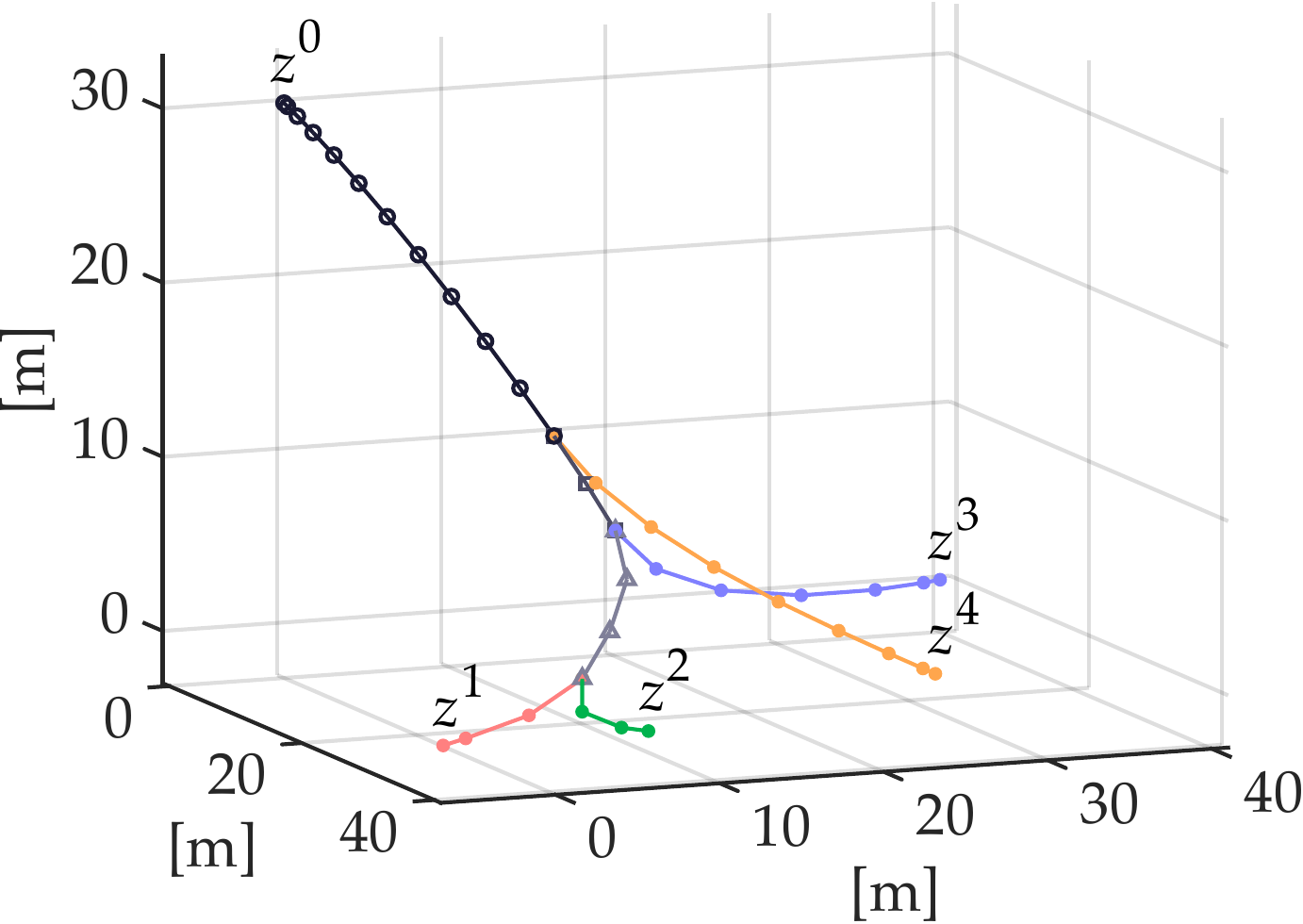}
    \caption{Position}
    \end{subfigure}
    
    \vspace{0.3cm}
     
    \begin{subfigure}[b]{\linewidth}
    \centering
    \includegraphics[width=0.85\linewidth]{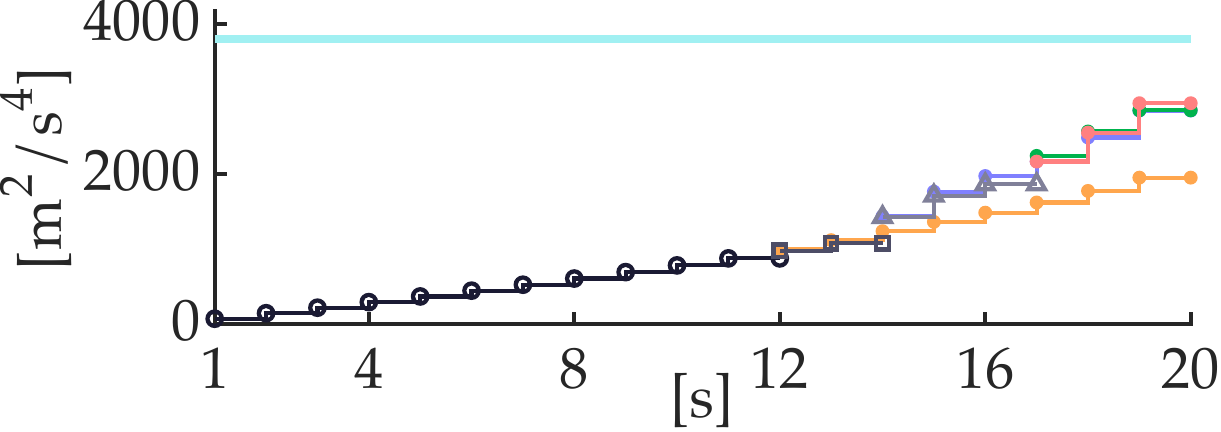}
    \caption{Cumulative trajectory cost}
    \end{subfigure}
    
    \vspace{0.3cm}
     
    \begin{subfigure}[b]{\linewidth}
    \centering
    \includegraphics[width=0.85\linewidth]{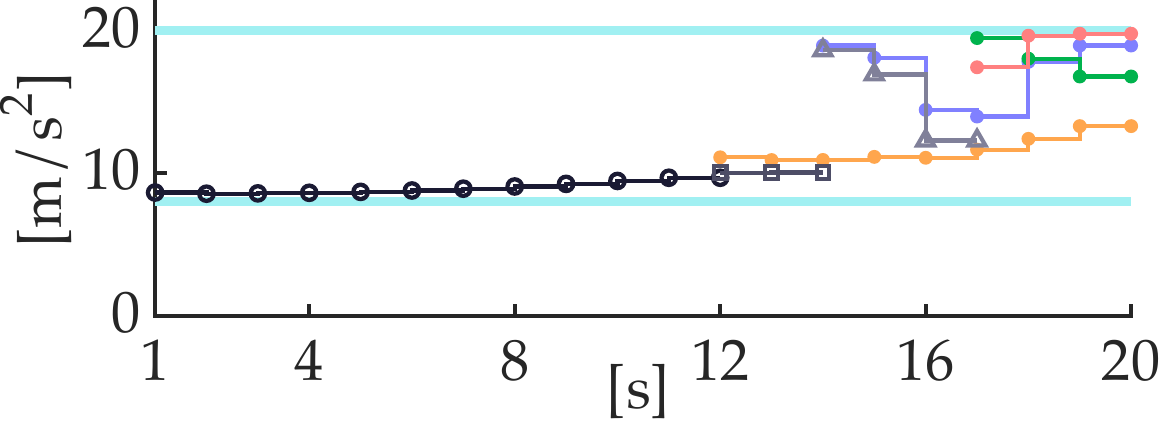}
    \caption{Thrust magnitude}
    \end{subfigure}
    
    \vspace{0.3cm}
     
    \begin{subfigure}[b]{\linewidth}
    \centering
    \includegraphics[width=0.85\linewidth]{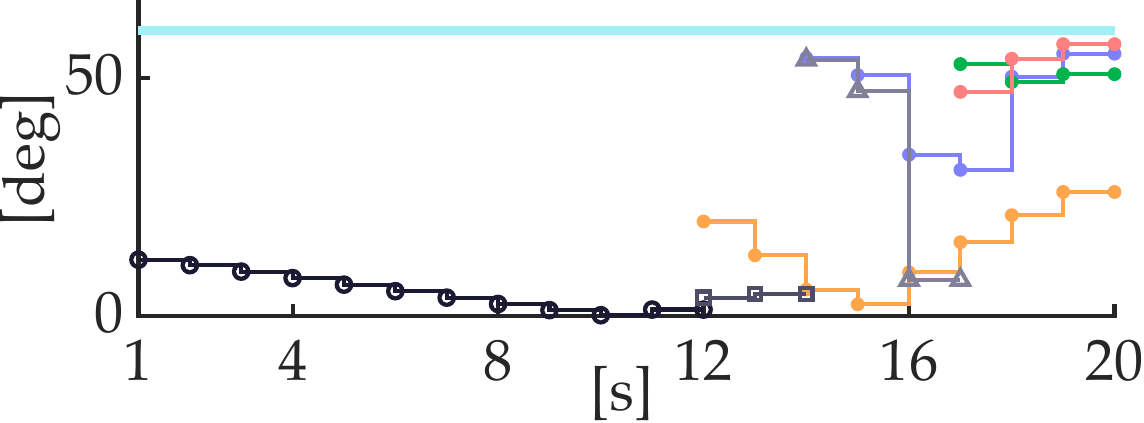}
    \caption{Thrust pointing angle}
    \end{subfigure}
\caption{Algorithm \ref{alg:ddtoqcvx} applied to the discrete-time convex optimal control example in Appendix \ref{app:ocp-eg-dt-cvx}.}
\label{fig:ddto-qcvx}
\end{figure}
\begin{figure}[!ht]
    \centering
    \includegraphics[width=0.8\linewidth]{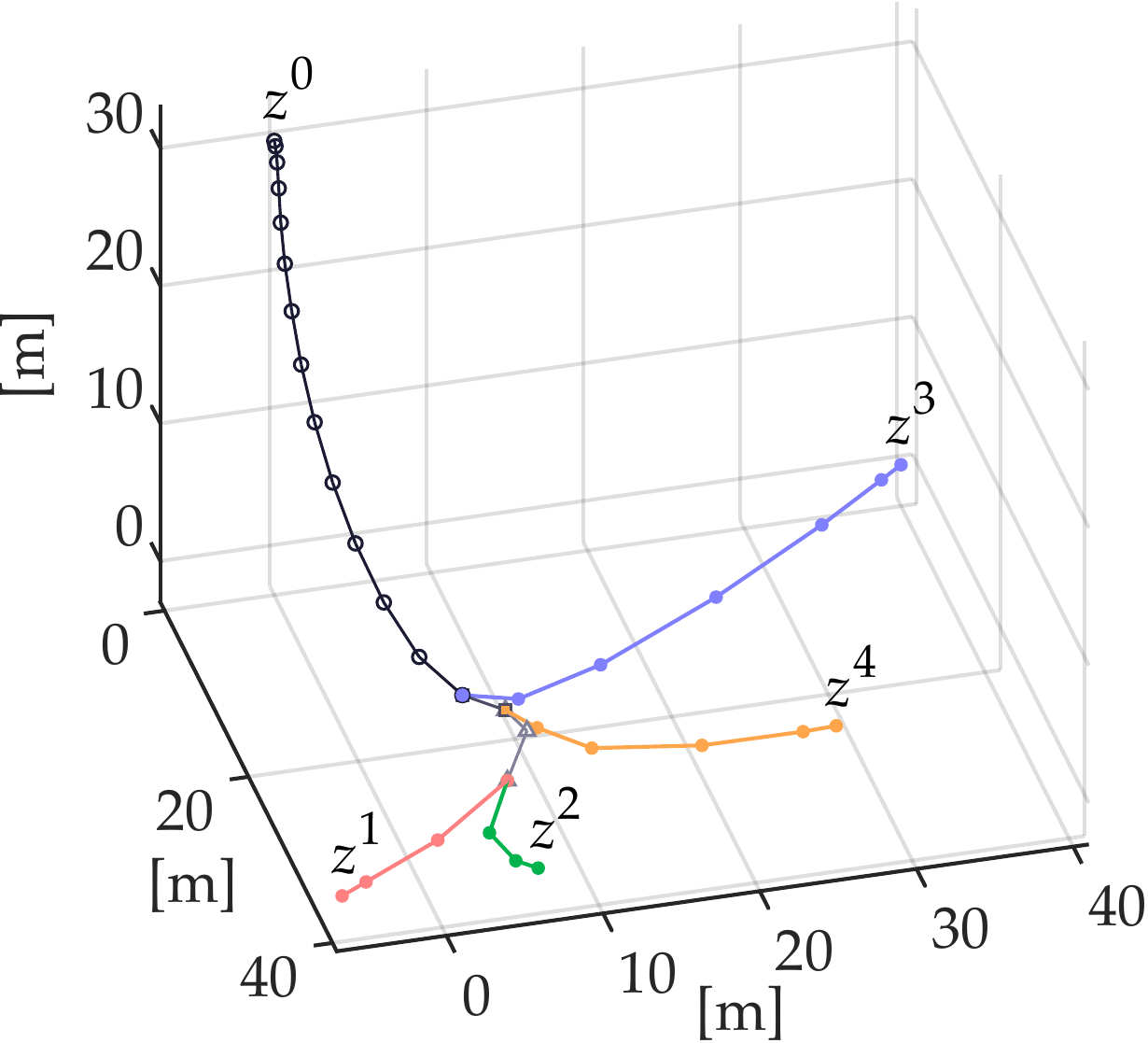}
    \caption{Position trajectories for the discrete-time convex optimal control example in Appendix \ref{app:ocp-eg-dt-cvx}, obtained with {\ddtomicp}.}
    \label{fig:ddto-micp}
\end{figure}
%
%
%
\begin{figure}[!ht]
\centering
\begin{subfigure}[b]{\linewidth}
\centering
\includegraphics[width=0.9\linewidth]{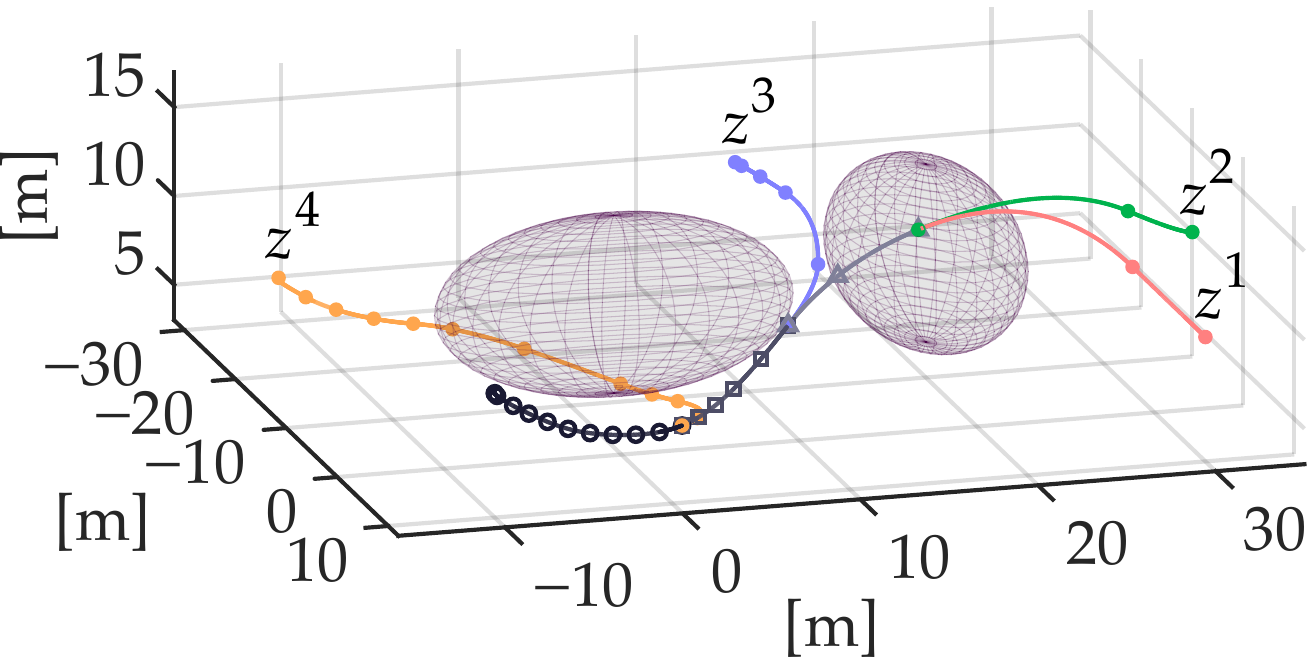}
\caption{Position}
\end{subfigure}

\vspace{0.3cm}
 
\begin{subfigure}[b]{\linewidth}
\centering
\includegraphics[width=0.85\linewidth]{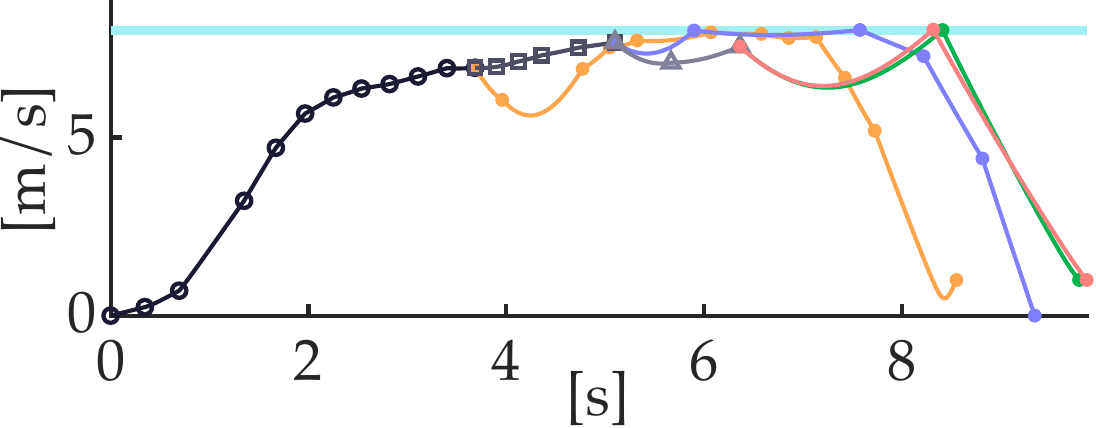}
\caption{Speed}
\end{subfigure}

\vspace{0.3cm}
 
\begin{subfigure}[b]{\linewidth}
\centering
\includegraphics[width=0.9\linewidth]{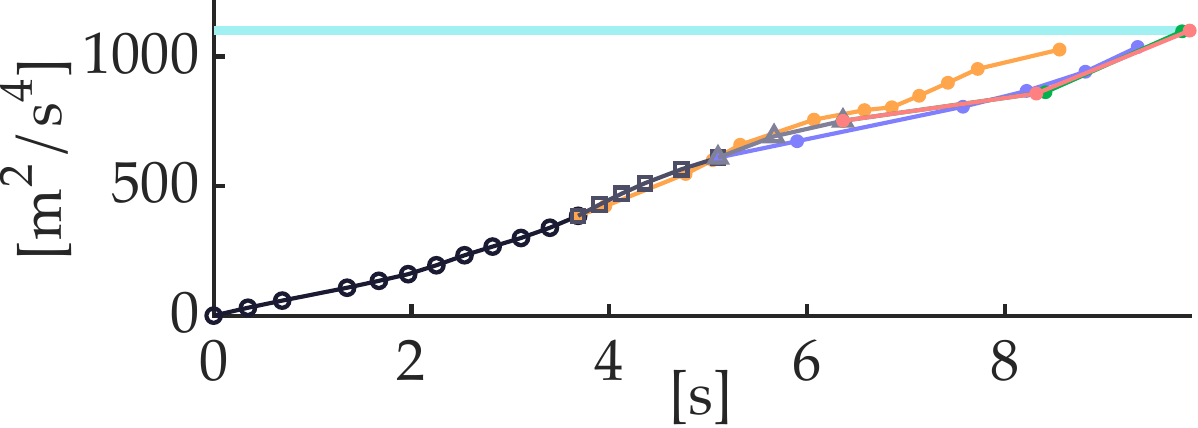}
\caption{Cumulative trajectory cost}
\end{subfigure}

\vspace{0.3cm}
 
\begin{subfigure}[b]{\linewidth}
\centering
\includegraphics[width=0.85\linewidth]{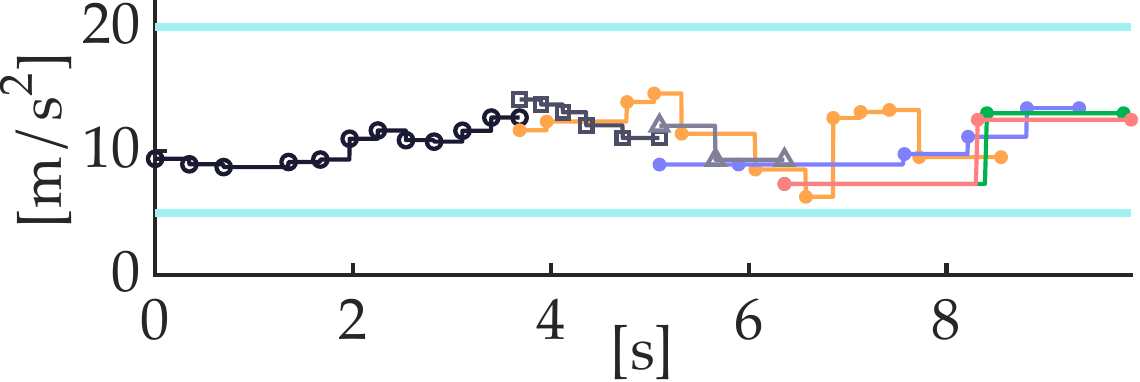}
\caption{Thrust magnitude}
\end{subfigure}

\vspace{0.3cm}
 
\begin{subfigure}[b]{\linewidth}
\centering
\includegraphics[width=0.85\linewidth]{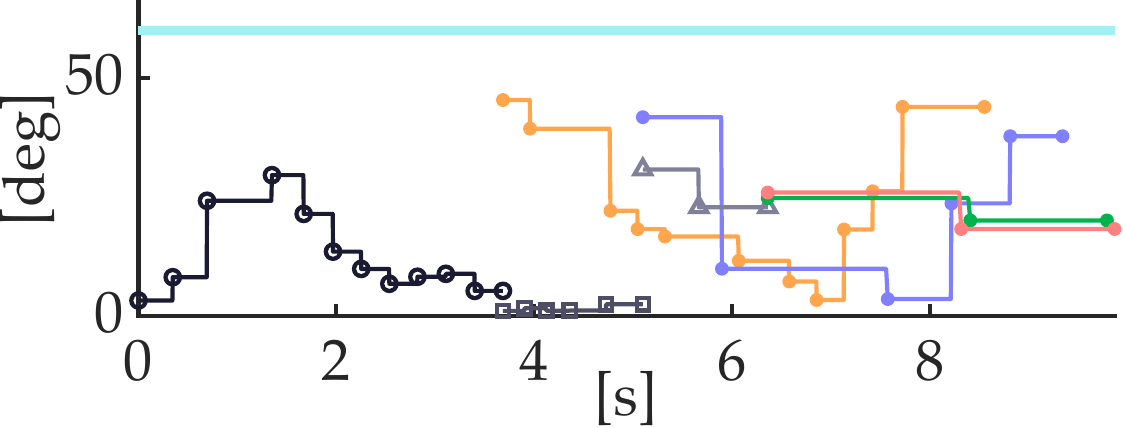}
\caption{Thrust pointing angle}
\end{subfigure}
\caption{Algorithm \ref{alg:ddtoscp} applied to the continuous-time nonconvex optimal control example in Appendix \ref{app:ocp-eg-ct-ncvx}.}
\label{fig:ddto-scp}
\end{figure}
\section{Conclusion}
We present a trajectory optimization framework called {\ddto}---deferred-decision trajectory optimization---for trajectory optimization in the presence of unmodeled uncertainties and contingencies due to imperfect knowledge of the system model or its environment. The key idea is to ensure that a collection of targets is available for as long as possible, which allows the vehicle to defer the decision to select a target as much as possible. In a closed-loop setting this would provide more time to quantify the uncertainties and contingencies so that the most reliable target can be eventually selected. To this end, we proposed a fully-deterministic optimization-based approach for formulating {\ddto} via constrained reachable sets and provided their equivalent representations as a cardinality minimization problems. Then by inferring the structure of the optimal {\ddto} solutions, we designed specialized solution methods for convex and nonconvex optimal control problems, which we demonstrated on two optimal control applications based on quadrotor motion planning. 

Future work will consider the complete closed-loop setting wherein useful information from measurement and perception is quantified and embedded into the {\ddto} formulation, with the goal of biasing the motion of the vehicle toward regions where the uncertainties and contingencies can be more easily quantified or detected.
%
\section*{Acknowledgments}
The authors gratefully acknowledge Samuel C. Buckner and Skye Mceowen for helpful discussions and their feedback on the initial draft of the paper.
\bibliographystyle{unsrturl}
\bibliography{references}
\section*{Appendix}
\appendix
\section{Time-Dilation}\label{app:dilation}
Consider a discrete-time dynamical system given by
\begin{align}
    x_{k+1} ={} & f_k(x_k,u_k),\quad k\ge 1.
\end{align}
We consider a trajectory and a control input sequence defined over the grid: $0=t_1<\ldots<t_N=\tfinal$. Suppose that the underlying continuous-time dynamical system is 
\begin{align}
    \dot{x}(t) = F(x(t),u(t)),\quad t\in[0,\tfinal]\label{eq:ct-dyn-sys}
\end{align}
and assume that the continuous-time control input $u$ is subject to a zero-order-hold parametrization. Then for each $k\in\bZr{1}{N-1}$, we have the following
\begin{align}
    f_k(x_k,u_k) \triangleq{} & x_k + \int_{t_k}^{t_{k+1}}F({}^kx(t),u_k)\mr{d}t,
\end{align}
where ${}^kx$ is a trajectory for \eqref{eq:ct-dyn-sys} over $[t_k,t_{k+1}]$ with constant control input $u_k$ and initial condition $x_k$. Note that subscript ``$k$'' in $f_k$ accounts for the possibility that the grid $t_1,\ldots,t_N$ is nonuniformly spaced. 

Next we apply time-dilation, which considers a new independent variable $\tau\in[0,1]$ by mapping $[0,1]$ to $[0,\tfinal]$ using a strictly-increasing, piecewise-linear continuous function $t$. Consider a uniformly-spaced grid within $[0,1]$ given by $0=\tau_1<\ldots<\tau_N=1$. For each $k\in\bZr{1}{N-1}$, the derivative of $t$, referred to as the dilation factor, is constant over the interval $[\tau_k,\tau_{k+1})$ and is denoted by $s_k$, i.e., 
\begin{align}
    & \frac{\mr{d}t(\tau)}{\mr{d}\tau} = s_k,\quad\forall\,\tau\in[\tau_k,\tau_{k+1}).
\end{align}
We treat the dilation factor as an additional control input, i.e., $\tilde{u}_k\triangleq (u_k,s_k)$ for each $k\in\bZr{1}{N-1}$, to obtain a new dynamical system with an augmented control input, i.e.,
\begin{align*}
    \tilde{f}(x_k,\tilde{u}_k) \triangleq{} & x_k + \int_{\tau_{k}}^{\tau_{k+1}}s_kF({}^kx(t(\tau)),u_k)\mr{d}\tau,\\
    ={} & f_k(x_k,u_k) = x_{k+1}.
\end{align*}
We constrain $s$ to be positive and bounded (for strict monotonicity and physically meaningful values of $t$). While all trajectories of horizon length $N$ for the new dynamical system are defined over $[0,1]$, they can each correspond to a different final time. Suppose that a collection of trajectories with same initial state but different target states are required to be coincident for as long as possible, then each of those trajectories can have a different final time in spite of having the same horizon length. The dilation factors of all trajectories will be the same in the coincident portion of trajectory; beyond that, they may differ (since the dilation factor serves a role similar to the original control inputs).
\section{Optimal Control Examples}\label{app:ocp-eg}
The following sections describe a discrete-time convex and a continuous-time nonconvex optimal control problem based on quadrotor motion planning.
\subsection{Discrete-Time Convex Problem}\label{app:ocp-eg-dt-cvx}
We consider a discrete-time model of a point-mass aerial vehicle where the state $x_k = (r_k,v_k)$ consists of three-dimensional position $r_k\in\bR^{3}$, velocity $v_k\in\bR^{3}$, and the control input $u_k\in\bR^3$ is an acceleration. The dynamical system is given by
\begin{align}
    x_{k+1} ={} & \underbrace{\begin{bmatrix} \eye{3} & \Delta t\eye{3} \\ \zeros{3\times3} & \eye{3}  \end{bmatrix}}_{A}x_k + \underbrace{\begin{bmatrix} \frac{\Delta t^2}{2}\eye{3} \\ \Delta t\eye{3}  \end{bmatrix}}_{B}u_k + \underbrace{\begin{bmatrix} \frac{\Delta t^2}{2}a \\ \Delta ta \end{bmatrix}}_{c},
\end{align}
where the vehicle has unit mass, $a\in\bR^{3}$ is the acceleration due to gravity, $\Delta t$ is the sampling time. The control input is subject to the following path constraints
\begin{subequations}
\begin{align}
    \|u_k\| \le{} & u_{\max},\\
    \|u_k\| \le{} & \sec\delta_{\max}\hat{e}^\top u_k,\label{eq:dt-pointing-cnstr}\\
    \hat{e}^\top u_k \ge{} & u_{\min},\label{eq:cvx-proxy-ctrl-lbnd}
\end{align}
\end{subequations}
where $u_{\max}$ and $u_{\min}$ are the upper and lower bounds, respectively, on the control input magnitude, $\delta_{\max}$ is the maximum angle between the control input and the pointing vector $\hat{e}$. Note that in the presence of the pointing constraint \eqref{eq:dt-pointing-cnstr}, a conservative convex approximation for the (nonconvex) lower bound constraint on the control input magnitude is provided by \eqref{eq:cvx-proxy-ctrl-lbnd}. For each $j\in J$, the targets are singleton sets denoted by $\mc{Z}^j = \{z^j\}$. The stage cost function $l$ in \eqref{eq:ddto-qcvx-cumcostcnstr} is given by
\begin{align}
    l(x_k,u_k) \triangleq{} & \|u_k\|^2.
\end{align} 
Table \ref{tab:dt-cvx-param} shows the parameter values chosen for the system, Algorithm \ref{alg:ddtoqcvx}, and {\ddtomicp} to generate the results in Figures \ref{fig:ddto-qcvx} and \ref{fig:ddto-micp}.
\begin{table}[!htpb]
    \centering
    \caption{}
    \label{tab:dt-cvx-param}
    \begin{tabular}{c|l}
    \hline
    Parameter & Value\\\hline
    $\Delta t$ & $0.5$ s\\
    $a$ & $(0,0,-9.806)$ m/s$^2$\\
    $n$ & $4$\\
    $N^j$ for $j\in\bZr{1}{n}$ & $20$\\
    $u_{\max},u_{\min}$ & $20,8$ m/s$^2$\\
    $\hat{e},\delta_{\max}$ & $(0,0,1),60^\circ$ \\
    $l_{\max}$ & $3794$ m$^2$/s$^4$\\
    $z^0$ & $(0,0,30,0,0,0)$ m,\,m/s\\
    $z^1$ & $(39.5,-6.25,0,0,0,0)$ m,\,m/s\\
    $z^2$ & $(39.5,6.25,0,0,0,0)$ m,\,m/s\\
    $z^3$ & $(28.3,28.3,0,0,0,0)$ m,\,m/s\\
    $z^4$ & $(40,0,0,0,0,0)$ m,\,m/s\\
    $\lambda^1,\lambda^2,\lambda^3,\lambda^4$ & $1,2,3,4$\\\hline
    \end{tabular}
\end{table}
\subsection{Continuous-Time Nonconvex Problem}\label{app:ocp-eg-ct-ncvx}
We consider a continuous-time model of a point-mass aerial vehicle where the state $x(t) = (r(t),v(t),\theta(t))$ consists of three-dimensional position $r(t)\in\bR^{3}$, velocity $v(t)\in\bR^3$, and cumulative trajectory cost $\theta(t)\in\bR$, and the control input $u(t)\in\bR^{3}$ is an acceleration. The dynamical system is given by
\begin{align}
    \dot{x}(t) ={} \!\!\begin{bmatrix} v(t) \\ u(t) - c_{\mr{d}}\|v(t)\|v(t) + a\\\|u(t)\|^2 \end{bmatrix}\!= {F(x(t),u(t))},
\end{align}
where the vehicle has unit mass, $a\in\bR^3$ is the acceleration due to gravity, and $c_{\mr{d}}$ is the drag coefficient. The path constraints on the vehicle are defined by $g:\bR^{7}\times\bR^{3}\to\bR^{7}$ as follows
\begin{align}
    g(x(t),u(t)) ={} & \begin{bmatrix}-\|H^1_{\mr{obs}}(r(t)-q^1_{\mr{obs}})\|^2+1\\[0.1cm]
                                      -\|H^2_{\mr{obs}}(r(t)-q^2_{\mr{obs}})\|^2+1\\[0.1cm]
                                       \|v(t)\|^2 - v_{\max}^2\\[0.1cm]
                                       \|u(t)\|^2 - u_{\max}^2\\[0.1cm]
                                      -\|u(t)\|^2 + u_{\min}^2\\[0.1cm]
                                       \|u(t)\|^2 - (\sec\delta_{\max}\hat{e}^\top u(t))^2\\
                                      -\hat{e}^\top u(t)                                        
                                \end{bmatrix}.
\end{align}
where $H^i_{\mr{obs}}$ and $q^i_{\mr{obs}}$, for $i=1,2$, are the shape matrices and centers, respectively, of ellipsoidal obstacles, $v_{\max}$ is the upper bound on speed, $u_{\max}$ and $u_{\min}$ are the upper and lower bounds, respectively, on the control input magnitude, and $\delta_{\max}$ is the maximum angle between the control input and the pointing vector $\hat{e}$. Note that $g_i$, for $i\in\bZr{1}{n_g}$ with $n_g=7$, in \eqref{eq:aug-sys-ct} denote the scalar-valued components of $g$. The second-order-cone constraint (due to pointing vector) is equivalently reformulated to a quadratic form to ensure continuous differentiability \cite[Sec. 3.2.4]{mosekcookbook}. The boundary condition constraint functions $P^j$ and $Q^j$ are defined as 
\begin{subequations}
\begin{align}
    P^j(x(\tfinal)) \triangleq{} & \theta(\tfinal) - l_{\max}, \\
    Q^j(x(\tfinal)) \triangleq{} & (r(\tfinal),v(\tfinal)) - z^j,
\end{align}
\end{subequations}
where $z^j\in\bR^6$ specify the target position and velocity, for $j\in J$, and $x(\tfinal)$ is the terminal state of a trajectory.

Table \ref{tab:ct-ncvx-param} shows the parameter values chosen for the system and Algorithm \ref{alg:ddtoscp} to generate the results in Figure \ref{fig:ddto-scp}.
\begin{table}[!htpb]
\centering
\caption{}
\label{tab:ct-ncvx-param}
\begin{tabular}{c|l}
\hline
Parameter & Value\\\hline
$a$ & $(0,0,-9.806)$ m/s$^2$\\
$c_{\mr{d}}$ & $0.01$ 1/m\\
$n,N$ & $4,23$\\
$v_{\max}$ & $8$ m/s\\
$u_{\max},u_{\min}$ & $20,5$ m/s$^2$\\
$\hat{e},\delta_{\max}$ & $(0,0,1),60^\circ$ \\
$l_{\max}$ & $1100$ m$^2$/s$^4$\\
$H^1_{\mr{obs}},H^2_{\mr{obs}}$ & $\mr{diag}(0.2,0.1,0.2),\mr{diag}(0.1,0.2,0.2)$\\[0.05cm]
$q^1_{\mr{obs}},q^2_{\mr{obs}}$ & $(-5,1,10),(-10,20,10)$ m\\[0.03cm]
$z^0$ & $(10,-10,10,0,0,0,0)$ m,\,m/s,\,m$^2$/s$^4$\\
$z^1$ & $(10,30,10,1,0,0)$ m,\,m/s\\
$z^2$ & $(-10,35,10,0,1,0)$ m,\,m/s\\
$z^3$ & $(-30,15,10,0,0,0)$ m,\,m/s\\
$z^4$ & $(-15,-15,10,0,1,0)$ m,\,m/s\\
$\lambda_1,\lambda_2,\lambda_3,\lambda_4$ & $1,2,3,4$\\
$\tilde{\mc{U}}^j$ & $\{u\in\bR^3|\|u\|_{\infty}\le u_{\max}\} \times [1,15]$\\
$\epsilon$ & $10^{-5}$\\\hline
\end{tabular}
\end{table}
\end{document}